\documentclass[10pt,leqno]{article}
\usepackage{a4wide}
\usepackage{amsmath}
\usepackage{amssymb}
\usepackage{amsthm}
\usepackage{enumerate}
\usepackage[dvips]{graphicx}
\usepackage{psfrag}
\usepackage{indentfirst}
\usepackage[margin=10pt,font=footnotesize]{caption}
\usepackage{color}

\renewcommand\today{\number\day\space \ifcase\month\or January\or February
    \or  March\or April\or May\or June\or July\or
       August\or September\or October\or
         November\or December\fi\space  \number\year}

\newcommand\g{{\gamma}}

\newcommand\e{{\varepsilon}}

\newcommand\boa{\mathboit a}

\newcommand\bof{\mathboit f}
\newcommand\bg{\mathboit g}

\newcommand\bn{\mathboit n}

\newcommand\bu{\mathboit u}
\newcommand\bv{\mathboit v}

\newcommand\bx{\mathboit x}

\newcommand\bU{\mathboit U}

\newcommand\beq{\begin{equation}}
\newcommand{\beqn}[1]{\begin{equation}\label{#1}}
\newcommand {\bsb}[1]{\begin{subequations}\label{#1}}
\newcommand\eeq{\end{equation}}
\newcommand {\esb}{\end{subequations}}
\newcommand{\bml}{\begin{multline}}
\newcommand{\eml}{\end{multline}}
\newcommand{\bal}{\begin{align}}
\newcommand{\eal}{\end{align}}
\newcommand{\bald}{\begin{aligned}}
\newcommand{\eald}{\end{aligned}}
\newcommand{\bgth}{\begin{gather}}
\newcommand{\egth}{\end{gather}}
\newcommand{\bgd}{\begin{gathered}}
\newcommand{\egd}{\end{gathered}}

\newcommand\barA{\vbox{\ialign{##\crcr\hbox{\kern .32em\vrule height 0pt width 3.5pt depth .28pt}\crcr\noalign
{\kern1.15pt\nointerlineskip}$\hfil\displaystyle{A}\hfil$\crcr}}}

\newcommand\barB{\vbox{\ialign{##\crcr\hbox{\kern .27em\vrule height 0pt width 4.2pt depth .28pt}\crcr\noalign
{\kern1.15pt\nointerlineskip}$\hfil\displaystyle{B}\hfil$\crcr}}}

\newcommand\barC{\vbox{\ialign{##\crcr\hbox{\kern .22em\vrule height 0pt width 4.5pt depth .28pt}\crcr\noalign
{\kern1.15pt\nointerlineskip}$\hfil\displaystyle{C}\hfil$\crcr}}}

\newcommand\barD{\vbox{\ialign{##\crcr\hbox{\kern .27em\vrule height 0pt width 4.5pt depth .28pt}\crcr\noalign
{\kern1.15pt\nointerlineskip}$\hfil\displaystyle{D}\hfil$\crcr}}}

\newcommand\barE{\vbox{\ialign{##\crcr\hbox{\kern .27em\vrule height 0pt width 4.5pt depth .28pt}\crcr\noalign
{\kern1.15pt\nointerlineskip}$\hfil\displaystyle{E}\hfil$\crcr}}}

\newcommand\barF{\vbox{\ialign{##\crcr\hbox{\kern .25em\vrule height 0pt width 4.5pt depth .28pt}\crcr\noalign
{\kern1.15pt\nointerlineskip}$\hfil\displaystyle{F}\hfil$\crcr}}}

\newcommand\barG{\vbox{\ialign{##\crcr\hbox{\kern .23em\vrule height 0pt width 4.5pt depth .28pt}\crcr\noalign
{\kern1.15pt\nointerlineskip}$\hfil\displaystyle{G}\hfil$\crcr}}}

\newcommand\barH{\vbox{\ialign{##\crcr\hbox{\kern .30em\vrule height 0pt width 4.5pt depth .28pt}\crcr\noalign
{\kern1.15pt\nointerlineskip}$\hfil\displaystyle{H}\hfil$\crcr}}}

\newcommand\barI{\vbox{\ialign{##\crcr\hbox{\kern .25em\vrule height 0pt width 3.0pt depth .28pt}\crcr\noalign
{\kern1.15pt\nointerlineskip}$\hfil\displaystyle{I}\hfil$\crcr}}}

\newcommand\barJ{\vbox{\ialign{##\crcr\hbox{\kern .32em\vrule height 0pt width 3.0pt depth .28pt}\crcr\noalign
{\kern1.15pt\nointerlineskip}$\hfil\displaystyle{J}\hfil$\crcr}}}

\newcommand\barK{\vbox{\ialign{##\crcr\hbox{\kern .30em\vrule height 0pt width 5.0pt depth .28pt}\crcr\noalign
{\kern1.15pt\nointerlineskip}$\hfil\displaystyle{K}\hfil$\crcr}}}

\newcommand\barL{\vbox{\ialign{##\crcr\hbox{\kern .22em\vrule height 0pt width 3.5pt depth .28pt}\crcr\noalign
{\kern1.15pt\nointerlineskip}$\hfil\displaystyle{L}\hfil$\crcr}}}

\newcommand\barM{\vbox{\ialign{##\crcr
\hbox{\kern .32em\vrule height 0pt width 6.0pt depth .28pt}\crcr\noalign
{\kern1.15pt\nointerlineskip}$\hfil\displaystyle{M}\hfil$\crcr}}}

\newcommand\barN{\vbox{\ialign{##\crcr\hbox{\kern .30em\vrule height 0pt width 4.7pt depth .28pt}\crcr\noalign
{\kern1.15pt\nointerlineskip}$\hfil\displaystyle{N}\hfil$\crcr}}}

\newcommand\barO{\vbox{\ialign{##\crcr\hbox{\kern .25em\vrule height 0pt width 4.2pt depth .28pt}\crcr\noalign
{\kern1.15pt\nointerlineskip}$\hfil\displaystyle{O}\hfil$\crcr}}}

\newcommand\barP{\vbox{\ialign{##\crcr\hbox{\kern .29em\vrule height 0pt width 4.2pt depth .28pt}\crcr\noalign
{\kern1.15pt\nointerlineskip}$\hfil\displaystyle{P}\hfil$\crcr}}}

\newcommand\barQ{\vbox{\ialign{##\crcr\hbox{\kern .25em\vrule height 0pt width 4.2pt depth .28pt}\crcr\noalign
{\kern1.15pt\nointerlineskip}$\hfil\displaystyle{Q}\hfil$\crcr}}}

\newcommand\barR{\vbox{\ialign{##\crcr\hbox{\kern .29em\vrule height 0pt width 4.2pt depth .28pt}\crcr\noalign
{\kern1.15pt\nointerlineskip}$\hfil\displaystyle{R}\hfil$\crcr}}}

\newcommand\barS
{\vbox{\ialign{##\crcr\hbox{\kern .23em\vrule height 0pt width 3.5pt depth .28pt}\crcr\noalign
{\kern1.15pt\nointerlineskip}$\hfil\displaystyle{S}\hfil$\crcr}}}

\newcommand\barT
{\vbox{\ialign{##\crcr\hbox{\kern .17em\vrule height 0pt width 5.0pt depth .28pt}\crcr\noalign
{\kern1.15pt\nointerlineskip}$\hfil\displaystyle{T}\hfil$\crcr}}}

\newcommand\barU
{\vbox{\ialign{##\crcr\hbox{\kern .20em\vrule height 0pt width 4.5pt depth .28pt}\crcr\noalign
{\kern1.15pt\nointerlineskip}$\hfil\displaystyle{U}\hfil$\crcr}}}

\newcommand\barV
{\vbox{\ialign{##\crcr\hbox{\kern .15em\vrule height 0pt width 5.0pt depth .28pt}\crcr\noalign
{\kern1.15pt\nointerlineskip}$\hfil\displaystyle{V}\hfil$\crcr}}}

\newcommand\barW{\vbox{\ialign{##\crcr
\hbox{\kern .13em\vrule height 0pt width 8.0pt depth .28pt}\crcr\noalign
{\kern1.15pt\nointerlineskip}$\hfil\displaystyle{W}\hfil$\crcr}}}

\newcommand\barX{\vbox{\ialign{##\crcr\hbox{\kern .25em\vrule height 0pt width 5.0pt depth .28pt}\crcr\noalign
{\kern1.15pt\nointerlineskip}$\hfil\displaystyle{X}\hfil$\crcr}}}

\newcommand\barY{\vbox{\ialign{##\crcr\hbox{\kern .12em\vrule height 0pt width 5.5pt depth .28pt}\crcr\noalign
{\kern1.15pt\nointerlineskip}$\hfil\displaystyle{Y}\hfil$\crcr}}}

\newcommand\barZ{\vbox{\ialign{##\crcr\hbox{\kern .28em\vrule height 0pt width 4.2pt depth .28pt}\crcr\noalign
{\kern1.15pt\nointerlineskip}$\hfil\displaystyle{Z}\hfil$\crcr}}}

\newcommand\barGamma{\vbox{\ialign{##\crcr\hbox{\kern .25em\vrule height 0pt width 4.5pt depth .28pt}\crcr\noalign
{\kern1.15pt\nointerlineskip}$\hfil\displaystyle{\varGamma}\hfil$\crcr}}}

\newcommand\barDe{\vbox{\ialign{##\crcr\hbox{\kern .33em\vrule height 0pt width 3.8pt depth .28pt}\crcr\noalign
{\kern1.15pt\nointerlineskip}$\hfil\displaystyle{\varDelta}\hfil$\crcr}}}

\newcommand\barTh{\vbox{\ialign{##\crcr\hbox{\kern .27em\vrule height 0pt width 4.2pt depth .28pt}\crcr\noalign
{\kern1.15pt\nointerlineskip}$\hfil\displaystyle{\varTheta}\hfil$\crcr}}}

\newcommand\barLa{\vbox{\ialign{##\crcr\hbox{\kern .30em\vrule height 0pt width 3.5pt depth .28pt}\crcr\noalign
{\kern1.15pt\nointerlineskip}$\hfil\displaystyle{\varLambda}\hfil$\crcr}}}

\newcommand\barXi{\vbox{\ialign{##\crcr\hbox{\kern .25em\vrule height 0pt width 4.5pt depth .28pt}\crcr\noalign
{\kern1.15pt\nointerlineskip}$\hfil\displaystyle{\varXi}\hfil$\crcr}}}

\newcommand\barPi{\vbox{\ialign{##\crcr\hbox{\kern .30em\vrule height 0pt width 4.5pt depth .28pt}\crcr\noalign
{\kern1.15pt\nointerlineskip}$\hfil\displaystyle{\varPi}\hfil$\crcr}}}

\newcommand\barSg{\vbox{\ialign{##\crcr\hbox{\kern .30em\vrule height 0pt width 4.7pt depth .28pt}\crcr\noalign
{\kern1.15pt\nointerlineskip}$\hfil\displaystyle{\varSigma}\hfil$\crcr}}}

\newcommand\barUp{\vbox{\ialign{##\crcr\hbox{\kern .17em\vrule height 0pt width 4.5pt depth .28pt}\crcr\noalign
{\kern1.15pt\nointerlineskip}$\hfil\displaystyle{\varUpsilon}\hfil$\crcr}}}

\newcommand\barPhi{\vbox{\ialign{##\crcr\hbox{\kern .22em\vrule height 0pt width 4.0pt depth .28pt}\crcr\noalign
{\kern1.15pt\nointerlineskip}$\hfil\displaystyle{\varPhi}\hfil$\crcr}}}

\newcommand\barPsi{\vbox{\ialign{##\crcr\hbox{\kern .20em\vrule height 0pt width 4.0pt depth .28pt}\crcr\noalign
{\kern1.15pt\nointerlineskip}$\hfil\displaystyle{\varPsi}\hfil$\crcr}}}

\newcommand\barOm{\vbox{\ialign{##\crcr\hbox{\kern .30em\vrule height 0pt width 4.2pt depth .28pt}\crcr\noalign
{\kern1.15pt\nointerlineskip}$\hfil\displaystyle{\varOmega}\hfil$\crcr}}}

\numberwithin{equation}{section}

\DeclareMathAlphabet{\bss}{OT1}{cmss}{bx}{n}
\DeclareMathAlphabet{\mathboit}{OT1}{cmr}{bx}{it}

\newtheorem{Theorem}{Theorem}[section]
\newtheorem{Proposition}[Theorem]{Proposition}
\newtheorem{Lemma}[Theorem]{Lemma}
\theoremstyle{definition}
\newtheorem{remark}[Theorem]{Remark}
\newcommand{\Id}{\mathrm{Id}}
\newcommand{\R}{\mathbb{R}}
\newcommand{\N}{\mathbb{N}}
\newcommand{\calH}{\mathcal{H}}
\newcommand\ba{\boldsymbol a}

\begin{document}
\author{David P.~Bourne$^1$, Sergio Conti$^2$ and Stefan M\"uller$^{2,3}$}
\date{22 December 2015}

\title{Energy Bounds for a Compressed Elastic Film on a Substrate}

\maketitle

\footnotetext[1]{Department of Mathematical Sciences, Durham University}
\footnotetext[2]{Institute for Applied Mathematics, University of Bonn}
\footnotetext[3]{Hausdorff Center for Mathematics, University of Bonn}

\begin{abstract}
We study pattern formation in a compressed elastic film which delaminates from a substrate. Our key tool is the
determination of rigorous  upper and lower bounds on the minimum value of a suitable energy functional.
The energy consists of two parts,  describing the
two main physical effects. The first part represents the elastic energy of the film, which is approximated using
 the von K\'arm\'an plate theory. The second part represents the fracture or delamination energy, which is approximated
 using the Griffith model of fracture.
 A simpler model containing the first term alone was previously studied
 with similar methods by
several authors,
assuming that the delaminated region is fixed. We include the fracture term,
transforming the elastic minimization into a free-boundary problem, and opening the way for patterns which
result from the interplay of elasticity and delamination.

After rescaling, the energy depends on only two parameters:
the rescaled film thickness, $\sigma$, and a measure of the bonding strength between the film and substrate, $\gamma$. We prove upper bounds on the minimum energy of the form
$\sigma^a \gamma^b$ and find that there are four different parameter regimes corresponding to different values of $a$ and $b$ and to different folding patterns of the film. In some cases the upper bounds are attained by self-similar folding patterns as observed in experiments.
Moreover, for two of the four parameter regimes we prove matching, optimal lower bounds.
\end{abstract}

%
%
\section{Introduction}
Compressed elastic sheets such as plastic films and fabric often exhibit self-similar folding patterns.  A typical example are folds in curtains, which decrease in number and increase in size from top to bottom as the folds merge. This coarsening phenomenon is also observed at the microscale in graphene and semiconductor films. The spontaneous delamination of prestrained semiconductor films from their substrates produces blisters with rich, self-similar folding patterns \cite{GioiaOrtiz98}. This also represents a challenge to manufacturers. Recently experimentalists have discovered how this, originally unwanted phenomenon, can be harnessed for thin film patterning and nanofabrication \cite{Ionov12}, for example to create nanotubes and nanochannels from prestrained semiconductor films \cite{CendulaKaravittaya2009}, \cite{MeiThurmer2007}.

In the mathematical community there is an ongoing programme to understand why these patterns occur. The calculus of variations has proved to be a useful tool, where the patterns are viewed as minimisers of an elastic potential energy. This is  the approach we take in this paper.

Motivated by the experiments of \cite{MeiThurmer2007}, we
study a variational model of
a two-layer material consisting of a rectangular elastic film on a substrate.
The film is clamped along one edge to the substrate and is free on the other three sides. Due to a lattice mismatch between the film and the substrate, the film suffers isotropic in-plane compression. It can relax this compression by delaminating from the substrate and buckling, but a short-range attractive force between the film and the substrate opposes this. We assign to the material the following energy:
\begin{equation}
\label{I.1}
I^{(\sigma,\g)} := I_{\textrm{vK}} + I_{\textrm{Bo}},
\end{equation}
where $I_{\textrm{vK}}$ is the von K\'arm\'an energy of isotropically compressed plates:
\begin{equation}
\label{I.2}
I_{\textrm{vK}}[\bu,w,l_1,l_2] := \int_{0}^{l_1} \! \int_{0}^{l_2} |D \bu + (D \bu)^\mathrm{T} + D w \otimes D w - \Id |^2 \, dx \, dy
+ (\sigma l_1)^2 \int_{0}^{l_1} \! \int_{0}^{l_2} | D^2 w|^2 \, dx \, dy
\end{equation}
and $I_{\textrm{Bo}}$ is a bonding energy that penalises delamination of the film from the substrate:
\begin{equation}
I_{\textrm{Bo}}[\bu,w,l_1,l_2] := \g \, | \{ (x,y)\in (0,l_1) \times (0,l_2) \, : \, w(x,y)>0 \} |.
\end{equation}
Here $(0,l_1) \times (0,l_2)$ is the set of material points of the film, $l_1\le l_2$, and $\bu(x,y)=(u(x,y),v(x,y))$ and $w(x,y)$ are the in-plane and vertical displacements of the film from an isotropically compressed state.
The substrate is taken to be at height $z=0$ so that film is bonded to the substrate at points where $w=0$.
The energy is rescaled so that $(u,v,w)=(x/2,y/2,0)$ corresponds to the stress-free, minimum energy state of the film, and $(u,v,w)=(0,0,0)$ has positive energy and corresponds to the isotropically compressed state. The energy has two parameters: $0<\sigma<1$ is the rescaled thickness of the film and $\g \ge 0$ is a measure of attractive force between the film and the substrate. We show how these are related to physical parameters such as the film thickness and Young's modulus in Appendix \ref{M}, where the energy \eqref{I.1} is derived.
The experimental setup is shown in Figure \ref{figbcm1a} and discussed below.
\begin{figure}[h]
\begin{center}
\includegraphics[width=0.6\textwidth]{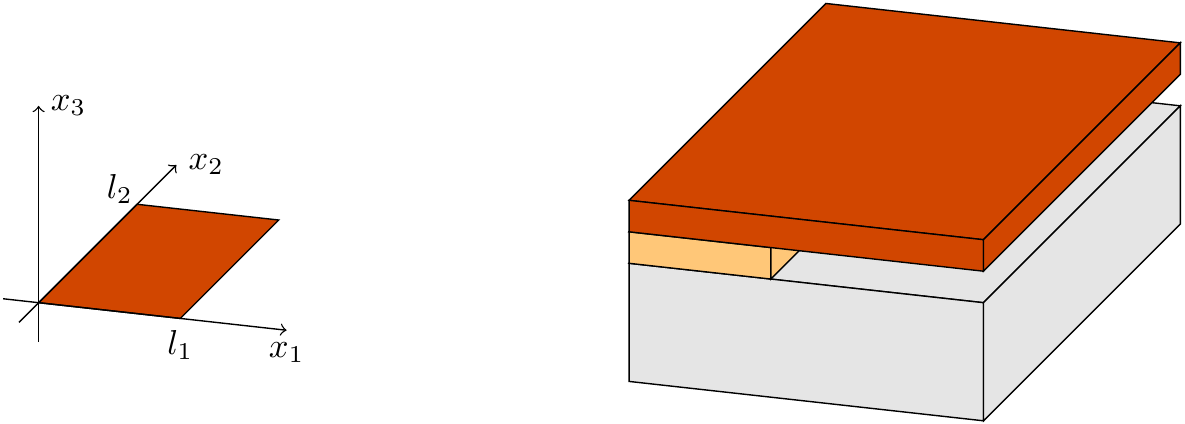}
\caption{Geometry of the partially delaminated film. The bottom layer is the substrate, the middle layer is the sacrificial buffer layer, which is removed by chemical etching in the region $x_1>0$, and the top layer is the thin elastic film.
The film is subject to compression at the boundary
$x_1=0$, where it is still attached to the buffer layer. It may rebond to the substrate in the region $x_1>0$. For simplicity, in our model we take the buffer layer to have zero thickness. Consequently we refer to the material as a two-layer material rather than a three-layer material. The general case is discussed briefly in Appendix \ref{H}.}
\label{figbcm1a}
\end{center}
\end{figure}

We assume that the film is clamped to the substrate along the edge $\{ 0 \} \times (0,l_2)$:
\beqn{I.4}
u(0,y)=0, \quad v(0,y)=0, \quad w(0,y) = 0, \quad D w (0,y) = \mathbf{0}
\eeq
and is free on the other three sides of the rectangle $(0,l_1)\times(0,l_2)$. Since the film cannot go below the substrate we also have the positivity constraint
\begin{equation}
w \ge 0.
\end{equation}

The von K\'arm\'an energy \eqref{I.2} is the sum of a stretching energy $I_{\textrm{S}}$, which penalises compression and extension, and a bending energy
$I_{\textrm{Be}}$:
\begin{equation}
\begin{aligned}
I_{\textrm{S}} & := \int_{0}^{l_1} \! \int_{0}^{l_2} |D \bu + (D \bu)^\mathrm{T} + D w \otimes D w - \Id |^2 \, dx \, dy, \\
I_{\textrm{Be}} & := (\sigma l_1)^2 \int_{0}^{l_1} \! \int_{0}^{l_2} | D^2 w|^2 \, dx \, dy.
\end{aligned}
\end{equation}
The folding patterns in the experiments of \cite{MeiThurmer2007}, and in compressed thin films in general, can be explained as the competition between the stretching and bending energies. The stretching energy has no minimum over the set of displacements satisfying the clamped boundary condition. Its infimum is zero, and an infimising sequence can be constructed by taking a periodic folding pattern, with folds perpendicular to the clamped boundary, and by sending the wavelength of the folds to zero. This relaxes the compression in the film but sends the bending energy to infinity. The competition between the stretching and bending energies determines the scale of the pattern.

We shall see that for some range of the parameters $(\sigma,\g)$ it is energetically favourable for the film to form a self-similar folding pattern with the wavelength of the folds increasing away from the clamped boundary.
Close to the boundary a fine folding pattern is needed in order to interpolate the folds to the clamped boundary condition without paying too much stretching energy. It would cost too much bending energy, however, to use such small folds in the whole domain and so branching occurs to obtain larger folds in the bulk of the domain.
The addition of the bonding energy, which is the novelty of this paper, complicates the situation further and gives a richer family of folding patterns.


\paragraph{Experimental Motivation.}
In the experiments of \cite{MeiThurmer2007} they use a three-layer material, where the bottom layer is a substrate (e.g., Si), the middle layer is a buffer layer (e.g., SiO$_2$), and the top layer is a thin semiconductor film (e.g., Si$_x$Ge$_{1-x}$). Due to a lattice mismatch between the film and the buffer layer, the thin film suffers isotropic, in-plane compression. See Figure \ref{figbcm1a}.

In the experiments, a slab of the buffer layer is removed by chemical etching. (An acid is used to eat away the buffer layer from the side, without damaging the film or substrate. Once the desired portion of the buffer layer has been removed the acid is washed out and the sample is dried.) This allows the thin film to partially relax the in-plane compression by folding. Since the buffer layer is thin, it is observed that the folds of the film come into contact with the substrate and bond to it via attractive interfacial forces. In this way submicro and nano scale channels are fabricated, which can be used, e.g., in nanofluidic devices. The patterns formed by the channels are self-similar, consisting of regions where the film is bonded to the substrate between folds that branch as they approach the boundary of the etched region.

The model above is a simplified model where we take the buffer layer to have zero thickness and treat the material as a two-layer material, although many of our results extend to the case of thin buffer layers as shown in Appendix \ref{H}.
Also, while our variational model is a major simplification of the dynamic etching, rinsing and drying processes used in experiments, we still obtain good qualitative agreement with experiments.
More sophisticated models of delamination appear for example in \cite{Annabattula} and \cite{BhattacharyaFonsecaFrancfort2002}.

\paragraph{Main results.}
The type of self-similar folding patterns seen in experiments are difficult to predict.
Typically the in-plane compression in the film is far past the critical buckling threshold and so standard buckling analysis (linear instability analysis)
 cannot predict the branching patterns. Minimising the energy $I^{(\sigma,\g)}$ numerically is also challenging since there are many local minimisers. Instead we construct approximate global minimisers of $I^{(\sigma,\g)}$ by hand. In the proof of Theorem \ref{UB1} we construct admissible displacement fields $(\bu^*,w^*)$ satisfying upper bounds of the form
\begin{equation}
\nonumber
I^{(\sigma,\g)}[\bu^*,w^*,l_1,l_2] \le \overline{c} \, l_1 l_2 \, \sigma^a \gamma^b
\end{equation}
for some $\overline{c}, a, b >0$. The powers $a$ and $b$ depend on the region in the phase diagram; in Theorem \ref{UB1} we identify four parameter regimes $A$--$D$ within the parameter space
$\{ (\sigma, \g) \in (0,1) \times [0,\infty) \}$ corresponding to different values of $a$ and $b$ and different folding patterns $(\bu^*,w^*)$ of the thin film.
These parameter regimes are shown in Fig.~\ref{Fig1} and the values of $a$ and $b$ are given in Theorem \ref{UB1}.

We will see that
the upper bound for regime $A$ is attained when the film is bonded to the substrate everywhere, for regime $B$ by a simple periodic folding pattern, and for regimes $C$ and $D$ by fold branching patterns.  What distinguishes regimes $C$ and $D$ is that in regime $C$ the film is bonded to the substrate in large parts of the domain. Also regime $C$ is bulk dominated in the sense that the order of the energy is determined by the deformation of the film in the interior of the domain, whereas regime $D$ is boundary dominated. Consequently we refer to regime $A$ as the \emph{flat regime}, regime $B$ as the \emph{laminate regime} (borrowing a term from the theory of phase transitions in metals), regime $C$ as the \emph{localised branching regime}, and regime $D$ as the \emph{uniform branching regime}.
Some of these patterns are shown in Figures \ref{RegimeB}--\ref{OneFold}.

For parameter regimes $A$ and $D$ we prove matching lower bounds:
\begin{equation}
\nonumber
\min_{(\bu,w)} I^{(\sigma,\g)}[\bu,w,l_1,l_2] \ge \underline{c} \, l_1 l_2 \, \sigma^a \gamma^b
\end{equation}
with the same $a$ and $b$ as in the upper bounds, but with a smaller multiplicative constant $0< \underline{c} < \overline{c}$. See Theorem \ref{LB1}. Therefore we obtain power laws of the form
\begin{equation}
\nonumber
\min_{(\bu,w)} I^{(\sigma,\g)}[\bu,w,l_1,l_2] \sim l_1 l_2 \, \sigma^a \gamma^b.
\end{equation}
See Remark \ref{opt}. This means that our approximate global minimisers $(\bu^*,w^*)$, while not being minimisers, do have the optimal energy scaling. Moreover we see that some of our optimal constructions $(\bu^*,w^*)$ exhibit self-similar folding patterns. In this sense we predict the patterns seen in experiments.

For regimes $B$ and $C$ it remains an open problem to prove matching upper and lower bounds.

\paragraph{Applications of the results.}
The upper bound constructions give scaling laws for the geometry of the patterns. For example, in parameter regime $C$ the upper bound construction corresponds to a self-similar folding pattern with periodicity cells of the form shown in Figure \ref{OneFold}, where the period $2h$ decreases towards the clamped boundary. The period $2h_0$, the fold width $2 \delta_0$ and the fold height $A_0$ of the coarsest pattern (the folds furthest from the clamped boundary) satisfy the scaling laws
\begin{equation}
h_0 \sim l_1 \sigma^{1/4} \g^{1/16}, \quad \delta_0 \sim l_1 \sigma^{3/4} \g^{-5/16}, \quad A_0 \sim l_1 \sigma^{1/2} \gamma^{-1/8}.
\end{equation}
These come from the proof of Proposition \ref{Prop:overall}. These scaling laws can be written in terms of the film thickness, Young's modulus, etc., using equation \eqref{M.7}. Moreover, by comparing these scaling laws with experiments, we could extract a value for the bonding strength $\g$, which is difficult to measure experimentally. The upper bound constructions could also be used as initial guesses for numerical simulations.

\paragraph{Related work.}
The variational study of pattern formation in compressed thin films was initiated by  Ortiz  and Gioia in the 90s \cite{OrtizGioia,GioiaOrtiz,GioiaOrtiz98},
and has meanwhile attracted significant attention in the mathematical literature.
Our results generalise those of \cite{BenBelgacem1} and \cite{JinSternberg}, who proved for the case $\g=0$ the optimal energy scaling $\min I^{(\sigma,0)} \sim \sigma$.
In \cite{BenBelgacem2} it is shown that the minimum energy scales the same way if the thin film is modelled using three-dimensional nonlinear elasticity, which justifies our choice of the von K\'arm\'an approximation (see also \cite{ContiDesimoneMueller2005}).
The scaling is different, however, if the in-plane displacements of the film $(u,v)$ are neglected, see \cite{GioiaOrtiz,OrtizGioia, JinKohn00}. The rigorous energy scaling approach used in this paper
was first used for the study of pattern formation in shape-memory alloys \cite{KohnMuller92,KohnMuller94}
and has proven successful in the study of a variety of other pattern formation problems, including for example
confinement of elastic sheets and crumpling patterns \cite{ContiMaggi}, and the structure of flux tubes in type-I superconductors \cite{ChoksiKohnOtto04,ChoksiContiKohnOtto2008}.
The limit in which the volume fraction of one phase is very small may lead to the occurrence of
a variety of phases with partial branching, as was demonstrated for superconductors in \cite{ChoksiKohnOtto04,ChoksiContiKohnOtto2008} 
and for shape-memory alloys and dislocation structures in \cite{B9_Zwicknagl,ContiZwicknagl}.
A finer mathematical analysis was possible in the case of annular thin films \cite{BellaKohn2014}.
The situation with compliant substrates leads to different patterns, which are homogeneous over the film,
see for example \cite{KohnNguyen13, BedrossianKohn2015} and references therein.
There is a large literature on folding patterns in compressed thin films and other approaches include linear instability analysis, post-buckling analysis and numerical methods, e.g.,
\cite{Audoly99,AudolyBoudaoud2003,AudolyRomanPocheau,CendulaKaravittaya2009,HuangIm,Pomeau98}.
These type of techniques were used by \cite{Annabattula} to study the same experiments \cite{MeiThurmer2007} as we do. Our results complement theirs; they use a different model and techniques to
obtain different types of results, namely quantitative predictions about the folding patterns away from the boundary, rather than focussing on self-similar branching as we do here.
In the physics literature self-similar folding patterns are referred to as wrinklons \cite{Vandeparre2001}.
A more detailed discussion of the literature  and an overview of the present results are given in the companion paper \cite{BourneContiMueller1}.

\paragraph{Outline of the paper.} In Section \ref{Main} we state our main results, the upper and lower bounds on the minimum value of the energy $I^{(\sigma,\g)}$.
The lower bounds are proved in Section \ref{LB} and the upper bounds are proved in Sections \ref{A}--\ref{E}.
The model is derived in Appendix \ref{M}.
In Appendix \ref{H} we show how our results can be extended to include the more general boundary conditions that correspond to the the experiments of \cite{MeiThurmer2007}. In Appendix \ref{App:P} we prove a version of the Poincar\'e inequality that is needed for the lower bounds.
%
%
\section{Main Results}
\label{Main}
In this section we state our main results. The proofs will be postponed to the following sections.
Let $\Omega := (0,l_1)\times(0,l_2)$ be the set of material points of the elastic film, $0<l_1\le l_2$, and let
\begin{equation}
\label{V}
V := \left\{
(\bu,w) \in H^1(\Omega;\R^ 2) \times H^2(\Omega) : w=0, \, \bu = D w = \mathbf{0} \textrm{ on } \{ 0 \} \times (0,l_2), \; w \ge 0
\right\}
\end{equation}
be the set of admissible displacements satisfying the clamped boundary condition and the positivity constraint.

\begin{Theorem}[Upper bounds]
\label{UB1}
Let $0 < l_1 \le l_2$. 
 Define the parameter space
$X = \{(\sigma,\g) \in (0,1) \times [0,\infty) \}$. Define parameter regimes
\begin{equation}
\nonumber
\begin{aligned}
A := & \{(\sigma,\g) \in X: \g \ge \sigma^{-1} \}, \\
B := & \{(\sigma,\g) \in X: \sigma^{-4/9} \le \g < \sigma^{-1} \}, \\
C := & \{(\sigma,\g) \in X: \sigma^{4/5} \le \g < \sigma^{-4/9} \}, \\
D := & \{(\sigma,\g) \in X: \g < \sigma^{4/5} \},
\end{aligned}
\end{equation}
see Figure \ref{Fig1}.
There exists a positive constant c, independent of $l_1$, $l_2$, $\sigma$, and $\g$, such that
\begin{equation}
\nonumber
\min_V I^{(\sigma,\g)} \le c \, l_1 l_2
\left\{
\begin{array}{ll}
1 & \textrm{if } (\sigma,\g) \in A, \\
(\sigma \g)^{2/5} & \textrm{if } (\sigma,\g) \in B, \\ 
\sigma^{1/2} \g^{5/8} & \textrm{if } (\sigma,\g) \in C, \\ 
\sigma & \textrm{if } (\sigma,\g) \in D.
\end{array}
\right.
\end{equation}
\end{Theorem}
\begin{proof}
The proof is done in  Sections \ref{A}--\ref{E}, separately for each regime.
Precisely, the bound in Regime A follows from Lemma \ref{Lemma:Flat},
the one in regime B from Proposition \ref{Prop:B}, regime C and D from Proposition \ref{Prop:overall}.
\end{proof}

\begin{figure}[h]
\centerline{\includegraphics[height=0.4\textheight]{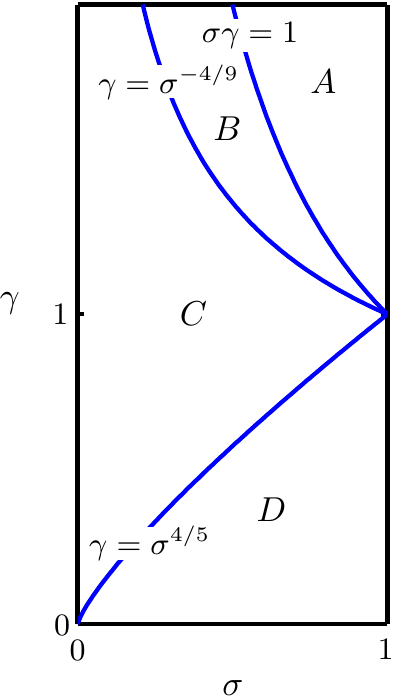}}
\caption{The four parameter regimes given in Theorem \ref{UB1}. $A$ is the flat regime, $B$ is the laminate regime, $C$ is the localised branching regime, and $D$ is the uniform branching regime.}
\label{Fig1}
\end{figure}
\noindent
Define an additional parameter regime
\begin{equation*}
D' :=  \{(\sigma,\g) \in X: \g < \sigma^{1/2} \} \supset D.
\end{equation*}

\begin{Theorem}[Lower Bounds]
\label{LB1}
 Let $\gamma \ge0$, $\sigma\in(0,1)$, $0<l_1\le l_2$.
 There exists a positive constant c, independent of $l_1$, $l_2$, $\sigma$, and $\g$, such that
 for all $(\bu,w)\in V$
 \begin{equation}
  I^{(\sigma,\gamma)}[\bu,w]\ge c \, l_1 l_2
\left\{\!\!\!
  \begin{tabular}{lll}
    $1$ &  if  $\gamma \ge \sigma^{-1}$ & \text{(regime $A$),}\\
    $(\sigma\gamma)^{2/3}$ &  if  $\sigma^{1/2}\le \gamma\le \sigma^{-1}$, 
    \\
    $\sigma$ &  if $\gamma<\sigma^{1/2}$.
     & \text{(regime $D'$).}
  \end{tabular}\right.
 \end{equation}
\end{Theorem}
\begin{proof}
This follows immediately from Lemma \ref{Lem:LB1} and Lemma \ref{LB2} in Section \ref{LB} below,
using the fact that
$(\sigma \gamma)^{2/3} \ge \sigma$ is equivalent to $\gamma \ge \sigma^{1/2}$.
\end{proof}

\begin{remark}[Optimality of the bounds in regimes $A$ and $D$]
\label{opt}
The bounds for regimes $A$ and $D$ are optimal in the sense that the lower and upper bounds scale the same way in the parameters $\sigma$ and $\gamma$:
\begin{equation}
\nonumber
\min_V I^{(\sigma,\g)} \sim
\left\{
\begin{array}{ll}
1 & \textrm{if } (\sigma,\g) \in A, \\
\sigma & \textrm{if } (\sigma,\gamma) \in D.
\end{array}
\right.
\end{equation}
Proving optimal bounds in the whole parameter space $X$ remains an open problem.
\end{remark}


\paragraph{Idea of the proof of Theorem \ref{UB1}.}
We describe the constructions used to obtain the upper bounds in Theorem \ref{UB1}, which correspond to different folding patterns of the thin film. The type of pattern is determined by the competition between the stretching, bending and bonding energies.

In the flat regime $A$ the bonding strength $\g$ is large compared to $\sigma$ and the upper bound $I^{(\sigma,\g)} \le c \,  l_1 l_2$ is obtained by taking $w=0$ everywhere, i.e., by taking the film to be bonded to the substrate everywhere.
See Lemma \ref{Lemma:Flat}.
In regimes $B$--$D$, where the bonding strength $\g$ is smaller, this upper bound can be improved by allowing the film to relax the in-plane compression by partially delaminating from the substrate and buckling.

In the laminate regime $B$ we obtain the upper bound $I^{(\sigma,\g)} \le c \,  l_1 l_2(\sigma \g)^{2/5}$ with a simple periodic folding pattern interpolated to the clamped boundary conditions, as shown in Figure \ref{RegimeB}. Note that the film is still bonded to the substrate in a large part of the domain.
See Proposition \ref{Prop:B}.
In regimes $C$ and $D$ 
a better upper bound can be achieved using a self-similar branching pattern, where the period of the folding pattern decreases towards the clamped boundary.

\begin{figure}[h]
\begin{center}
\includegraphics*[width=\textwidth,trim=0 100 0 140]{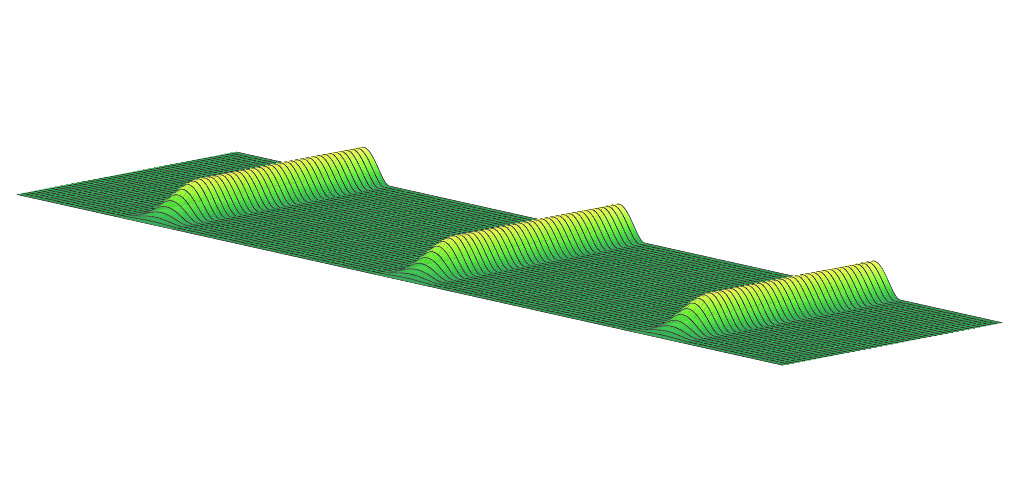}
\caption{\label{RegimeB} The upper bound construction for the laminate regime $B$. Away from the clamped boundary the film forms a periodic pattern consisting of folds in between regions where it is bonded to the substrate. This is interpolated to the clamped boundary conditions in a thin boundary layer.}
\end{center}
\end{figure}

In the subset of regime $D$ where $\gamma \le \sigma$ the bonding energy is a lower order term and the upper bound
$I^{(\sigma,\g)} \le c \,  l_1 l_2 \sigma$ can be obtained using the same period-halving construction that was used for the case $\g=0$ in
 \cite{BenBelgacem1}, as shown in Figure \ref{RegimeG}. Note that the film is delaminated from the substrate almost everywhere.

\begin{figure}[h]
\begin{center}
\includegraphics*[width=0.9\textwidth]{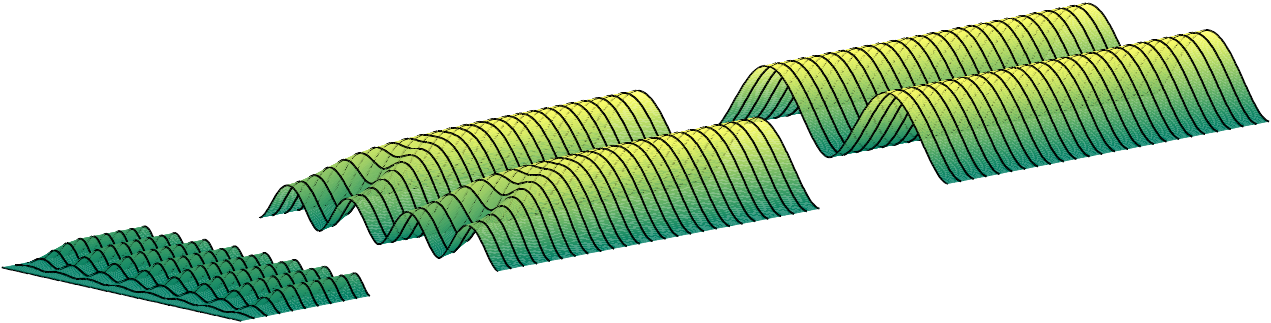}
\caption{\label{RegimeG} The upper bound construction for parameter regime $\{ \gamma \le \sigma\} \subset D$. Away from the clamped boundary the film forms a periodic folding pattern (right). The folds undergo a sequence of period-halvings (centre) before being interpolated to the clamped boundary (left). The film is delaminated from the substrate almost everywhere, but we do not pay too much bonding energy for this since the bonding strength $\g$ is small compared to $\sigma$.}
\end{center}
\end{figure}

In the rest of regime $D$ and in regime $C$ a new upper bound construction is needed.
Away from the clamped boundary we take the film to form a periodic pattern, with periodicity cells of the form shown in Figure \ref{OneFold}. It is energetically favourable for the period of the pattern to decrease towards the clamped boundary. By construction the stretching energy of each periodicity cell is zero. By requiring that the bending energy of each periodicity cell equals its bonding energy, we find that fold width $2 \delta$ is related to the period $2h$ by $\delta \sim h^{1/3}$. Therefore, as we halve the period, $h \mapsto 2^{-1} h$, we must multiply the fold width by a factor of $2^{-1/3} \approx 0.8$. This means that the relative area $\frac{\delta}{h}$ where the film is delaminated in the periodicity cell \emph{increases} as $h$ decreases. In order to achieve this we perform each period halving $h \mapsto 2^{-1} h$ using a two step refinement procedure consisting of fold shrinkage (Figure \ref{RegimeE}, right) and fold splitting and separation (Figure \ref{RegimeE}, left).
(For a real film the two branching steps would happen simultaneously, but this does not change the order of the energy.)
This takes one cell of period $2h$ and fold width $2 \delta$ and produces two cells of period $h$ and fold width $2^{2/3} \delta \approx 1.6 \delta$.

 We repeat this refinement process until the pattern is fine enough to be interpolated to the clamped boundary conditions without paying too much stretching energy.
 We will see in the proof of Proposition \ref{Prop:overall} that if $\gamma \ge 1$, then refinement stops before $\delta=h$, i.e., before the film becomes delaminated from the substrate almost everywhere.

\begin{figure}[h]
\begin{center}
\includegraphics*[width=0.9\textwidth]{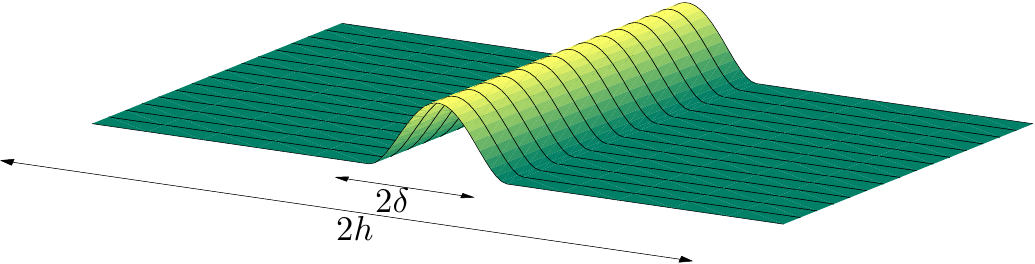}
\caption{\label{OneFold} One of the periodicity cells used for regimes $B$, $C$, $D$. A fold of width $2 \delta$ lies between two regions where the film is bonded to the substrate.}
\end{center}
\end{figure}

\begin{figure}[h]
\begin{center}
\includegraphics[width=0.9\textwidth]{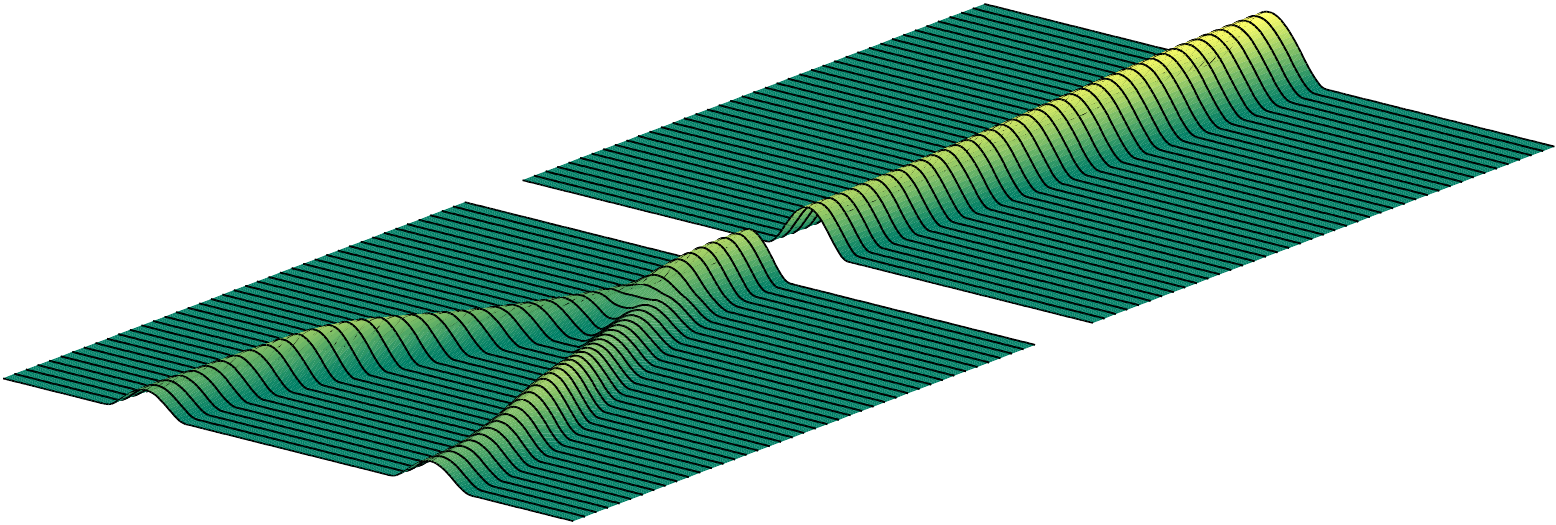}
\caption{\label{RegimeE} The branching construction used for regimes $C$ and $D$. Each refinement consists of two steps: fold shrinkage (right), where a fold is reduced in width and height, and fold splitting (left), where a fold splits into two folds of the same width.}
\end{center}
\end{figure}

If $\gamma < 1$, however, the pattern refinement continues until $\delta=h$ and the film is delaminated from the substrate almost everywhere.
  At this point the period of the folds is still too large; if refinement were stopped here then the stretching energy in the interpolation boundary layer would be too high. Therefore we continue refining the pattern towards the clamped boundary by period-halving, $\delta \mapsto 2^{-1} \delta$, using the
  same construction as for regime $\{ \gamma \le \sigma\}$, shown in Figure \ref{RegimeG}. This continues until $\delta=l_1 \sigma$ at which point the pattern is fine enough to be interpolated to the clamped boundary conditions. See the proof of Proposition \ref{Prop:overall}.



%
%

\section{Lower Bound Proofs}
\label{LB}
In this section we prove the main ingredients of the lower bound in
Theorem  \ref{LB1}.
\begin{Lemma}[Lower bound for $\gamma \ge 0$]
\label{Lem:LB1}
Let $\gamma \ge 0$, $\sigma \in (0,1)$, $0< l_1\le l_2$. Then for all $(\bu,w)\in V$
\beq
\nonumber
I^{(\sigma,\g)}[\bu,w,l_1,l_2] \ge c \, l_1 l_2 \sigma.
\eeq
\end{Lemma}
\begin{proof}
By scaling it suffices to consider the case $l_1=1$. Further, it suffices to prove the bound for $l_2=1$,
and apply it to each of the squares $(0,1)\times (k,k+1)$, $k\in \N\cap [0, l_2-1]$, and use that
$\lfloor l_2\rfloor\ge l_2/2$ for all $l_2\ge 1$.

The bound for $l_1=l_2=1$ and $\g=0$ was proven in \cite[Lemma 1]{BenBelgacem1} and implies the general case $\g \ge 0$ since
$I^{(\sigma,\g)} \ge I^{(\sigma,0)}$.
\end{proof}

\begin{Lemma}[Lower bounds for $\gamma \ge \sigma^2$]
\label{LB2}
Let $\sigma \in (0,1]$, $\gamma \ge \sigma^2$.
Then for all $(\bu,w)\in V$
\beq
\nonumber
I^{(\sigma,\g)}[\bu,w,l_1,l_2] \ge c \, l_1 l_2 \left\{
\begin{array}{ll}
1 & \textrm{if } \gamma \ge \sigma^{-1},
\\
(\sigma \gamma)^{2/3} & \textrm{if } \sigma^2 \le \gamma \le \sigma^{-1}.
\end{array}
\right.
\eeq
\end{Lemma}
\begin{proof}
As above, it suffices to consider the case $l_1=l_2=1$. Let $l \in (0,1]$. Subdivide $\Omega:=(0,1)^2$ into $\lceil l^{-1} \rceil^2$ squares
of side length $l$, denoted by $q_{ij}$:
\begin{equation*}
q_{ij} := l(i,i+1) \times l(j,j+1) \; \cap \; (0,1)^2, \qquad i,j \in \{ 0, \ldots , \lceil l^{-1} \rceil -1 \}.
\end{equation*}
We say that $q_{ij}$ is \emph{good} if
\begin{equation*}
|\{w=0\} \cap q_{ij}| \ge \frac{l^2}{2}.
\end{equation*}
Otherwise $q_{ij}$ is \emph{bad}.
Let $N_B$ be the number of bad squares, $\Omega_G$ be the union of all the good squares, and $\Omega_B$ be the union of the bad ones.
Let $E$ be the energy of a given deformation $(\bu,w)$.
It is easy to see that
\begin{equation*}
N_B \frac{l^2}{2} \; \le \; |\Omega_B \cap \{ w>0 \}| \; \le \; |\{ w>0 \}| \; \le \; \frac{E}{\gamma},
\end{equation*}
which implies
\begin{equation}
\label{NBb}
N_B \le \frac{2 E}{\gamma l^2}.
\end{equation}
Let $q_{ij}$ be a good square. By the Poincar\'e inequality
\begin{equation}
\label{eq:Poin}
\int_{q_{ij}} |D w|^2 \, d \bx \le l^2 \int_{q_{ij}} |D^2 w|^2 \, d \bx.
\end{equation}
(The  Poincar\'e constant $l^2$ is obtained in Appendix \ref{App:P}.)
Summing over all good squares yields
\begin{equation}
\label{b1}
\int_{\Omega_G} |Dw|^2 \, d \bx \le l^2 \int_{\Omega_G} |D^2 w|^2 \, d \bx \le \frac{l^2}{\sigma ^2} E.
\end{equation}
Define
\begin{equation*}
\eta := \int_{\Omega_G} | D\bu + D\bu^T - \Id | \, d \bx.
\end{equation*}
We claim that
\begin{equation}
\label{eta}
\eta \le E^{1/2} + \frac{l^2}{\sigma ^2} E.
\end{equation}
Proof: First note that
\begin{equation}
\label{ineq1}
E \ge \int_{\Omega} |D \bu + D \bu^T - \Id + Dw \otimes Dw |^2 \, d \bx \ge \left( \int_{\Omega} |D \bu + D \bu^T - \Id + Dw \otimes Dw | \, d \bx \right)^2
\end{equation}
by the Cauchy-Schwarz inequality since $\Omega$ has area 1.
By the definition of the Frobenius norm
\begin{equation*}
| D w \otimes Dw |^2 = w_x^4 + 2 w_x^2 w_y^2 + w_y^4 = (w_x^2 + w_y^2)^2 = |Dw|^4.
\end{equation*}
Therefore
\begin{equation}
\label{ineq2}
\int_{\Omega_G} | D w \otimes Dw | \, d \bx = \int_{\Omega_G} |Dw|^2 \, d \bx \le \frac{l^2}{\sigma ^2} E
\end{equation}
by equation \eqref{b1}.
Using the triangle inequality and combining \eqref{ineq1} and \eqref{ineq2} gives
\begin{equation*}
\eta \le \int_{\Omega_G} |D \bu + D \bu^T - \Id + Dw \otimes Dw | \, d \bx + \int_{\Omega_G} | D w \otimes Dw |\, d \bx
\le E^{1/2} + \frac{l^2}{\sigma ^2} E
\end{equation*}
as required.

At this point we define $\bU(\bx):=(2\bu(\bx)-\bx)\chi_{\Omega_G}(\bx)$. Then $\bU\in SBV(\Omega;\R^2)\subset SBD(\Omega)$
and its jump set $J_\bU$ is contained in the union of the boundaries of the bad squares, hence
$\calH^1(J_\bU)\le 4 l N_B\le 8E/\gamma l$ by (\ref{NBb}). At the same time, $\|e(\bU)\|_{L^1(\Omega)}=\eta$.
By the Korn-Poincar\'e inequality for $SBD$ functions proven
in \cite[Th. 1]{ChambolleContiFrancfort} (with $p=1$) we obtain that there are a set $\omega_\Gamma\subset\partial\Omega$
with $\calH^1(\omega_\Gamma)\le c\calH^1(J_\bU)$
and
an affine function $\ba$ with $D\ba+D\ba^T=0$ such that
$\int_{\{0\}\times (0,1)\setminus \omega_\Gamma} |\ba(\bx)-\bx| d\calH^1(\bx)\le c\eta$.
Since $a_2$ is constant on this set,
if the length of $\omega_\Gamma$ is less than $1/2$ we obtain
 $\eta\ge c$ for some universal constant $c$. Conversely,
if the length of $\omega_\Gamma$ is more than $1/2$ we obtain
$1/2\le c  \calH^1(J_\bU)\le c E/\gamma l$. Therefore
 at least one of $E\ge c\gamma l$ and $ \eta\ge c$
holds. Recalling (\ref{eta}) this gives
\begin{equation}
\label{LBalmost}
E \ge c \min \left\{ \gamma l, 1, \frac{\sigma^2}{l^2} \right\}.
\end{equation}
We consider different choices of $l$. Equating the second two terms on the right-hand side of \eqref{LBalmost} gives $l = \sigma$ and
\begin{equation*}
E \ge c \min \{ \sigma \gamma, 1 \} = c \quad \textrm{if} \quad \sigma \gamma \ge 1.
\end{equation*}
Equating the first and third terms on the right-hand side of \eqref{LBalmost} gives $l = \sigma^{2/3} \gamma^{-1/3}$ and
\begin{equation*}
E \ge c \min \{ (\sigma \gamma)^{2/3}, 1 \} = c (\sigma \gamma)^{2/3}  \quad \textrm{if} \quad \sigma \gamma \le 1.
\end{equation*}
Since we require that $l \in (0,1]$, then this estimate is only valid if in addition $\sigma^{2/3} \gamma^{-1/3}\le 1 \Longleftrightarrow \gamma \ge\sigma^2$.
This completes the proof.

In order to make the argument more transparent we also give a self-contained proof, which does not apply
the rigidity estimate in  \cite[Th. 1]{ChambolleContiFrancfort} but instead makes direct
usage of some ideas from its proof, which can be much simplified in the present context.

\begin{figure}
\centerline{\includegraphics[height=0.2\textheight]{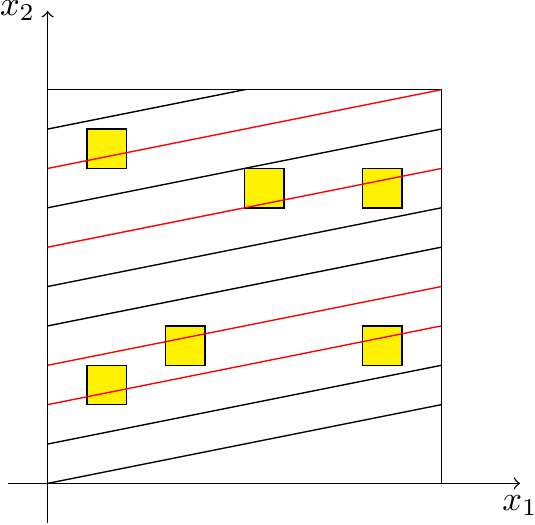}}
\caption{Sketch of the geometry in the proof of Lemma \ref{LB2}. The squares in $\Omega_B$ are shown,
as well as some of the lines $(0,y)+\bn \R$.}
\label{Fig7}
\end{figure}

We assume for now that
\begin{equation}
\label{Eub}
E \le \frac{1}{16} \gamma l.
\end{equation}
Therefore by equations \eqref{NBb} and \eqref{Eub}
\begin{equation}
\label{numBad}
N_B l \le \frac{1}{8}.
\end{equation}

Pick $\bn := (\cos \theta, \sin \theta) \in S^1$ with $n_1 \ge \frac{1}{\sqrt{2}}$, equivalently $\theta \in [-\pi/4,\pi/4]$.
Define
\begin{equation*}
\omega_{\bn} := \Big\{ y \in (0,1) \, : \,
\left\{ (0,y) + \bn \mathbb{R}\right\} \cap \Omega_B = \emptyset \Big\},
\end{equation*}
see Figure \ref{Fig7}.
Note that $\mathcal{L}^1 (\omega_{\bn}^c)$ is maximised, e.g., when $\theta = -\pi /4$ and
the union of bad squares $\Omega_B$ is given by
\begin{equation*}
 \bigcup_{i=0}^{N_B-1} (0,l) \times l(2i,2i+1).
\end{equation*}
In this case each line
 $(0,y) + \bn \mathbb{R}$ intersects at most one bad square and
 $\mathcal{L}^1 (\omega_{\bn}^c)=2 l N_B$.
By combining this and \eqref{numBad} we obtain
\begin{equation*}
\mathcal{L}^1 (\omega_{\bn}) \ge 1 - 2l N_B
\ge \frac{3}{4}
\end{equation*}
for all admissible $\bn$.
For $y \in \omega_\bn$ let $I_\bn(y) := \left\{ (0,y) + \bn \mathbb{R}\right\} \cap (0,1)^2$ and
\begin{equation*}
\Omega_\bn := \bigcup_{y \in \omega_\bn} I_\bn (y).
\end{equation*}
We estimate
\begin{equation}
\label{Omn}
\int_{\Omega_\bn} | D \bu + D \bu^T - \Id | \, d \bx \le \int_{\Omega_G} | D \bu + D \bu^T - \Id | \, dx = \eta.
\end{equation}
The function $f_n(\bx) := 2 \bn \cdot \bu(\bx) - \frac{x_1}{n_1}$ vanishes at $x_1=0$ since $\bu$ satisfies the
boundary condition $\bu(0,y)=(0,0)$. Also
\begin{equation}
\label{Dfn}
Df_\bn \cdot \bn = 2 \bn \cdot D \bu \cdot \bn - 1 = \bn \cdot (D \bu + D \bu^T - \Id) \bn.
\end{equation}
By the Poincar\'e inequality, \eqref{Dfn} and \eqref{Omn}
\begin{equation}
\label{genP}
\int_{\Omega_\bn} | f_\bn | \, d \bx \le c \int_{\Omega_\bn} | Df_\bn \cdot \bn | \, d \bx
\le c \int_{\Omega_\bn} | D \bu + D \bu^T - Id | \, d \bx \le c \eta
\end{equation}
for some $c>0$. Fix three different, admissible values of $\bn$, called $\bn$, $\bn'$, $\bn''$, with corresponding angles $\theta$, $\theta'$,
$\theta''$. Set $\tilde{\Omega}=\Omega_\bn \cap \Omega_{\bn'} \cap \Omega_{\bn''}$. Then
\begin{equation*}
| \tilde{\Omega} | > \frac{1}{64},
\end{equation*}
where the number on the right-hand side is obtained by taking
$\theta = \pi/4$, $\theta'=-\pi/4$, $\theta''=0$, $\omega_\bn = [1/4,1]$, $\omega_{\bn'} = [0,3/4]$, $\omega_{\bn''} = [0,3/8] \cup [5/8,1]$.

Given $x_1 > 0$, consider the following over-determined linear system for $\bv \in \mathbb{R}^2$:
\begin{equation*}
2 \bn \cdot \bv = \frac{x_1}{n_1}, \quad 2 \bn' \cdot \bv = \frac{x_1}{n'_1}, \quad 2 \bn'' \cdot \bv = \frac{x_1}{n''_1},
\end{equation*}
which can be written as
\begin{equation}
\label{ls}
2
\begin{pmatrix}
\bn \\ \bn' \\ \bn''
\end{pmatrix}
\bv
= x_1
\begin{pmatrix}
1/n_1 \\ 1/n_1' \\ 1/n_1''
\end{pmatrix}
\end{equation}
if we consider $\bn$, $\bn'$, $\bn''$ to be row vectors.
It is easy to check that the null-space of the adjoint of the matrix on the left-hand side is spanned by $(\sin(\theta'-\theta''),\sin(\theta''-\theta'),\sin(\theta-\theta'))$. Therefore \eqref{ls} has a solution $\bv$ if and only if
\begin{align*}
0 & = \begin{pmatrix}
1/n_1 \\ 1/n_1' \\ 1/n_1''
\end{pmatrix}
\cdot
\begin{pmatrix}
\sin(\theta'-\theta'') \\ \sin(\theta''-\theta') \\ \sin(\theta-\theta')
\end{pmatrix}
\\
& = \frac 12
\big(
\cos^2 \theta (\sin 2 \theta'' - \sin 2 \theta')
+
\cos^2 \theta' (\sin 2 \theta - \sin 2 \theta'')
+
\cos^2 \theta'' (\sin 2 \theta' - \sin 2 \theta)
\big).
\end{align*}
Now choose $\theta>0$, $\theta'=-\theta$. Then the right-hand side is nonzero since $\theta'' \ne \pm \theta$, and there is no solution $\bv$ to \eqref{ls}.
Therefore for this choice of $\bn$, $\bn'$, $\bn''$
\begin{equation}
\label{minv}
\min_{v \in \mathbb{R}^2} \left|2 \bn \cdot \bv - \frac{x_1}{n_1}\right| + \left|2 \bn' \cdot \bv - \frac{x_1}{n'_1}\right| + \left|2 \bn'' \cdot \bv - \frac{x_1}{n''_1}\right| > 0
\end{equation}
for any $x_1 >0$ and uniformly if $x_1 \ge x_1^*> 0$. By \eqref{eta}, \eqref{genP}, \eqref{minv}
\begin{align}
\nonumber
E^{1/2} + \frac{l^2}{\sigma^2} E & \ge \eta \\
\nonumber
& \ge \frac{1}{3c} \int_{\tilde{\Omega} \cap \{ x_1 \ge x_1^*\}} |f_\bn| + |f_{\bn'}| + |f_{\bn''}| \, d \bx \\
\nonumber
& \ge \frac{1}{3c} |\tilde{\Omega} \cap \{ x_1 \ge x_1^*\}|
\min_{\substack{v \in \mathbb{R}^2 \\ x_1 \in [x_1^*,1]}} \left|2 \bn \cdot \bv - \frac{x_1}{n_1}\right| + \left|2 \bn' \cdot \bv - \frac{x_1}{n'_1}\right| + \left|2 \bn'' \cdot \bv - \frac{x_1}{n''_1}\right|
\\
\label{3.14}
& \ge \tilde{c}
\end{align}
for some constant $\tilde{c}>0$.
Putting together \eqref{Eub} and \eqref{3.14} leads to (\ref{LBalmost}), and the proof is concluded as above.
\end{proof}

\section{Upper Bound Construction for the Flat Regime $A$}
\label{A}
In this section we prove Theorem \ref{UB1} for parameter regime $A$.
In this regime $\g$ is very large and the upper bound is obtain by taking the film to be bonded to the substrate everywhere.
\begin{Lemma}[Energy of the flat construction]
\label{Lemma:Flat}
There exists a constant $c>0$ such that for all $\sigma \in (0,1)$, $\gamma > 0$
\begin{equation}
\label{A.1}
\min_V I^{(\sigma,\gamma)} \le  c \, l_1 l_2.
\end{equation}
 \end{Lemma}
 \begin{proof}
The displacement field $(u,v,w) = (0,0,0)$ satisfies the desired upper bound \eqref{A.1}.
Note that a smarter choice of displacement field is  $(u,v,w) = (x/2,0,0)$, which yields the same bound \eqref{A.1} but with a smaller constant $c$. Here the film relaxes compression in the $x$-direction, the direction orthogonal to the clamped boundary, by spreading out.
\end{proof}
%
%
%
%
%
%
%
%
%
%
\section{Upper Bound Construction for the Laminate Regime $B$}
\label{B}
In this section we prove Theorem \ref{UB1} for parameter regime $B$.
In this regime it is favourable for the film to buckle, but not to exhibit branching patterns. See Figure \ref{RegimeB}.

In this and the next Section we work in components, so we write $\bx=(x,y)$ and
$\bu=(u,v)$, the stretching energy takes the form
\begin{equation}\label{eqiscomponents}
I_{\textrm{S}} = \int_\Omega (2 u_x + w_x^2 - 1)^2 \, d \bx
+ 2 \int_\Omega (u_y + v_x + w_x w_y)^2 \, d \bx
+ \int_\Omega (2 v_y + w_y^2 -1)^2 \, d \bx.
\end{equation}
\subsection{Construction in the Interior}
\label{Bs1}
First we construct a displacement field on an interior rectangle of $(0,l_1) \times (0,l_2)$, away from the clamped boundary $\{x=0\}$. By translation invariance of the energy we can take the rectangle to be $(0,l)\times(-h,h)$ with $l < l_1$, $2h \le l_2$.
Take $\delta < h$ and consider a vertical displacement of the form
\beqn{B.2}
w(x,y) := w(y) :=
\left\{
\begin{array}{cl}
0 & y \in (-h,-\delta], \\
\displaystyle \frac{A}{2} \left(1+\cos \frac{\pi y}{\delta} \right) & y \in [-\delta,\delta], \\
0 & y \in [\delta,h).
\end{array}
\right.
\eeq
This describes an $x$-independent displacement where the film has one fold of height $A$ and width $2 \delta$, and is bonded down to the substrate elsewhere, see Figure \ref{OneFold}.
Choosing $u$ and $v$ so that the stretching energy vanishes, i.e.,  $I_\mathrm{S} =0$, and with $v(x,-h)=v(x,h)=0$
yields
\beqn{B.3}
\begin{aligned}
u(x,y) & := u(x) := \frac 12 x + d, \\
v(x,y) & := v(y) :=
\left\{
\begin{array}{cl}
\smallskip
\displaystyle \frac 12 (y+h) & y \in (-h,-\delta], \\
\smallskip
\displaystyle \frac{y}{2} - \frac{A^2 \pi^2}{16 \delta^2} \left(y - \frac{\delta}{2 \pi} \sin \frac{2 \pi y}{\delta} \right)
& y \in [-\delta,\delta],
\\
\displaystyle \frac 12 (y-h) & y \in [\delta,h),
\end{array}
\right.
\end{aligned}
\eeq
for any constant $d$.
In order for $v$ to be continuous we must have
\beqn{B.4}
A^2 = \frac{8}{\pi^2}  \delta h.
\eeq
It is a simple calculus exercise to verify the following:
\begin{Lemma}[Energy of the basic laminate construction]
\label{Energy:Bint}
The displacement $(u,v,w)$ defined in equations \eqref{B.2}--\eqref{B.4} has the following energy on the rectangle $(0,l)\times(-h,h)$:
\begin{equation*}
\label{B.5}
I^{(\sigma,\gamma)} = 2 \pi^2 (\sigma l_1)^2 \frac{lh}{\delta^2} + 2 \g l \delta.
\end{equation*}
\end{Lemma}

%
%
\subsection{Construction in the Boundary Layer}
\label{Bs2}
 We define a displacement field
$(\tilde{u},\tilde{v},\tilde{w})$ on a boundary layer rectangle $(0,\varepsilon) \times (-h,h)$ by interpolating between the clamped boundary conditions \eqref{I.4} and the displacements introduced in equations \eqref{B.2}--\eqref{B.4}: Let
\begin{equation}
\label{B.7}
\tilde{u}(x) := \frac 12 x, \quad
\tilde{v}(x,y) := \psi^2\left( \frac{x}{\varepsilon} \right) v(y), \quad
\tilde{w}(x,y) := \psi \left( \frac{x}{\varepsilon} \right) w(y),
\end{equation}
where $\psi$ is the cubic interpolating polynomial satisfying $\psi(0)=0$, $\psi(1)=1$, $\psi'(0) = \psi'(1) = 0$, i.e.,
\begin{equation*}
\label{B.8}
\psi(t) := t^2(3-2t).
\end{equation*}
Note that in \eqref{B.7} we take $\tilde{v} = \psi^2 v$ rather than the more natural choice of $\tilde{v} = \psi v$ since then the stretching term $\tilde{w}_y^2 + 2 \tilde{v}_y -1$ is of order $1$ rather than of order $h/\delta$. See equation \eqref{B.13}.

We extend $(\tilde{u},\tilde{v},\tilde{w})$ to the whole boundary layer $(0,\varepsilon) \times (0,l_2)$ by gluing together $l_2/2h$ copies of $(\tilde{u},\tilde{v},\tilde{w})$ along their horizontal boundaries. (If $2h$ does not divide $l_2$, then we glue together $\lfloor l_2/2h \rfloor$
copies to get a construction on $(0,\varepsilon) \times (0,2h \lfloor l_2/2h \rfloor)$ and define $u = x/2$, $v = w = 0$ on $(0,\varepsilon) \times (2h \lfloor l_2/2h \rfloor,l_2)$. For simplicity we assume that $2h$ divides $l_2$ in the rest of the paper since it does not affect the energy bound.)
\begin{Lemma}[Energy of the boundary layer construction]
\label{Energy:BBL}
Let $0<\delta \le \varepsilon$, $\delta\le h\le l_2$.
The displacement field $(\tilde{u},\tilde{v},\tilde{w}):(0,\varepsilon) \times (0,l_2) \to \mathbb{R}^3$ defined above satisfies the following energy bound:
\beqn{B.16}
I^{(\sigma,\gamma)} \le C l_2 \left( \frac{h^2}{\e} + \e + (\sigma l_1)^2 \frac{\e}{\delta^2} +\g \frac{\e \delta}{h} \right).
\eeq
\end{Lemma}
\begin{proof}
Observe that $w$, defined in equation \eqref{B.2}, satisfies the following: 
\begin{equation*}
\label{B.9}
|\partial^m_y w| \le C h^{1/2} \delta^{1/2-m} \quad \textrm{for } m \ge 0.
%
\end{equation*}
Thus
$\tilde{w}$ satisfies
\beqn{B.10}
| \partial^n_x \partial^m_y \tilde{w} | \le C h^{1/2} \delta^{1/2-m} \varepsilon^{-n} \quad \textrm{for } m \ge 0, \textrm{ } n \ge 0.
\eeq
The displacement $v$, defined in equation \eqref{B.3}, satisfies
$|v| \le C h$
and so
\beqn{B.12}
|\tilde{v}_x|  \le C \frac{h}{\varepsilon}.
\eeq
Recall that $v$ was chosen so that $w_y^2 + 2v_y-1=0$. Therefore 
\beqn{B.13}
|\tilde{w}_y^2+2\tilde{v}_y-1| = |\psi^2(w_y^2 + 2v_y-1) - (1-\psi^2)| =|1-\psi^2| \le 1.
\eeq
From equations \eqref{B.10}--\eqref{B.13} we can estimate the stretching, bending, and bonding energy in the boundary layer
$(0,\varepsilon) \times (0,l_2)$:
\beqn{B.15}
I_{\textrm{S}} \le C l_2 \left[ \frac{h \delta^3}{\varepsilon^3} + \frac{h^2}{\varepsilon} + \varepsilon \right], \quad
I_{\textrm{Be}} \le  C \frac{l_2}{h} (\sigma l_1)^2  \varepsilon \delta \left[ \frac{h \delta}{\varepsilon^4} + \frac{h}{\delta \varepsilon^2} + \frac{h}{\delta^3} \right],  \quad
I_{\textrm{Bo}} \le  C \frac{l_2}{h} \g  \varepsilon \delta.
\eeq
%
By using the assumptions $\delta \le \e$ and $\delta \le h$, the bounds \eqref{B.15} reduce to \eqref{B.16}, as required.
\end{proof}
\subsection{Complete Construction}
\label{Bs3}
We are now in a position to prove the upper bound for the laminate regime $B$:
\begin{Proposition}[Energy of the laminate construction]
\label{Prop:B}
Let $0<l_1 \le l_2$, $\sigma \g \le 1$, $\g \ge 1$. Then
\begin{equation*}
\min_V I^{(\sigma,\gamma)} \le C l_1 l_2 (\sigma \g)^{2/5}.
\end{equation*}
\end{Proposition}
\begin{proof}
Take the displacement field $(u,v,w)$ that was defined in \eqref{B.2}--\eqref{B.4} on $(0,l)\times(-h,h)$ and extend it to the domain $(\varepsilon,l_1) \times (0,l_2)$
by taking $l=l_1-\varepsilon$, $d=\e /2$, and gluing together $l_2/2h$ copies along their horizontal boundaries (assuming without loss of generality that $2h$ divides $l_2$ as above).
Lemma \ref{Energy:Bint}
implies that this construction has energy
\beqn{B.6}
I^{(\sigma,\g)} \le C \frac{(l_1-\e) l_2}{h} \left[ (\sigma l_1)^2 \frac{h}{\delta^2} + \g \delta \right].
\eeq
Define $(u,v,w)$ on $(0,\varepsilon) \times (0,l_2)$ using the boundary layer construction from Section \ref{Bs2}. By combining \eqref{B.16} and \eqref{B.6}
we find that the energy on the whole domain $(0,l_1) \times (0,l_2)$ satisfies
\beqn{B.17}
\begin{aligned}
I^{(\sigma,\g)} & \le C l_2 \left[
\frac{h^2}{\e} + \e + (\sigma l_1)^2 \frac{\e}{\delta^2} + \g \frac{\e \delta}{h} + (\sigma l_1)^2 \frac{(l_1-\e)}{\delta^2} +
\g \frac{(l_1-\e) \delta}{h}
\right] \\
& =C  l_2
\left[ \frac{h^2}{\e} + \e + (\sigma l_1)^2 \frac{l_1}{\delta^2} + \g \frac{l_1 \delta}{h} \right].
\end{aligned}
\eeq
Now we chose $\e$, $h$, and $\delta$ to minimise the order of the energy. This can be done by equating
terms on the right-hand side of \eqref{B.17}, which yields
\beqn{B.18}
h = l_1 (\sigma \g)^{2/5}, \quad \e = h, \quad \delta = (\sigma l_1)^{2/3} \g^{-1/3} h^{1/3} = l_1 \sigma^{4/5} \g^{-1/5}.
\eeq
Note that the conditions $l_1 \le l_2$, $\sigma \g \le 1$, $\g \ge 1$
ensure that the geometrical restrictions $\delta \le h$, $h \le l_2$, $\e \le l_1$ and the constraint $\delta \le \e$ are satisfied.
Substituting \eqref{B.18} into \eqref{B.17} gives
\begin{equation*}
I^{(\sigma,\g)} \le  l_1 l_2 (\sigma \g)^{2/5}
\end{equation*}
as required.
\end{proof}
%
%
%
%
%
%
%
%
%
%
\section{Upper Bound Constructions for the Branching Regimes $C$ and $D$}
\label{E}
In this section we prove Theorem \ref{UB1} for the branching regimes $C$ and $D$. The basic idea is to use copies of a folding pattern of the form shown in Figure \ref{OneFold} and to decrease $h$ and $\delta$ towards the clamped boundary by a sequence of fold branchings.
\subsection{Branching constructions.}
\label{BranchingConstructions}
First we prove a general lemma about the cost of fold branching. This covers both our fold splitting construction (Lemma \ref{LemmaFSS}) and our fold shrinkage construction (Lemma \ref{Lemm:FS}).
\begin{Lemma}[The cost of branching]
\label{LemmaUB}
  Let $\Omega=(0,l)\times(-h,h)$, $h \le l$. Let $0 < \delta \le h$ and let $w\in C^4(\Omega;[0,\infty))$ satisfy the
  following:
  \begin{itemize}
  \item[(i)] $w(x,y)=0$ for $y$ in a neighbourhood of $\pm h$; $w_x(x,y)=0$
    for $x$ in a neighbourhood of  $0$, $l$;
  \item[(ii)] \label{lemmaubabspart} $w(x,-y)=w(x,y)$;
  \item[(iii)] \label{lemmaubabsavg} For all $x \in (0,l)$
  \begin{equation*}
  \int_{0}^h (1- w_y^2) \, dy = 0;
  \end{equation*}
\item[(iv)] \label{lemmaubabsder} For all $a,b \in \{0,1,2\}$,
  \begin{equation}
    |\partial^a_x \partial^b_y w|\le  c\left(\frac{h}{\delta l}\right)^a
\left(\frac{1}{\delta}\right)^b (h\delta)^{1/2}.
  \end{equation}
  \end{itemize}
  Define
\begin{equation}
\label{uv}
   v(x,y):=\frac12 \int_{-h}^y (1-w_y^2) \, d \tilde{y}, \qquad u(x,y):= \frac{1}{2} x - \int_{-h}^y (w_x w_y + v_x) \, d \tilde{y}.
\end{equation}
Then $v=0$ for $y=\pm h$ and $u=\frac 12 x$ on $\partial \Omega$, and $v(x,-y)=-v(x,y)$.

Assume additionally either
\begin{itemize}
  \item[(v)]\label{lemmaubabsint1} $w(x,y)=0$ whenever $|y|> c \delta$
\end{itemize}
or
\begin{itemize}
  \item[(vi)]\label{lemmaubabsint2} $\delta \le h/2$ and
  $w(x,y) = A(x) [\psi_\delta (y + \varphi(x)) + \psi_\delta (y - \varphi(x))]$,
  where $\psi_\delta$ is an even function supported on the interval $[-\delta,\delta]$ satisfying $|\psi_\delta^{(k)}| \le C / \delta^{k}$
   for $k\in\{0,1,2\}$,   and where $\varphi:[0,l] \to [0,h/2]$
and
$|\varphi^{(k)}|\le C h/l^k$ for $k\in\{0,1,2\}$.
\end{itemize}
Then
\begin{equation}
\label{eq:EbB}
  I^{(\sigma,\gamma)}[\bu,w,\Omega]\le  c \gamma \delta l + c (\sigma l_1)^2 \frac{l h }{\delta^2}
  + c \frac{h^6}{\delta l^3}\,.
\end{equation}
\end{Lemma}

\begin{remark}[Assumptions of Lemma \ref{LemmaUB}]
\label{Remark:assump}
Recall that the stretching energy of the film  was given in (\ref{eqiscomponents}).
The definition of $u$ in \eqref{uv} ensures that the second term of the stretching energy vanishes, and the definition of $v$ ensures that the third term vanishes.
  \end{remark}

\begin{proof}
First we prove the boundary conditions on $v$. The condition $v(x,-h)=0$ is clear from the definition of $v$.
 Assumptions (ii) and (iii) imply
\begin{equation}\label{eq:(iii)'}
\int_{-h}^h (1- w_y^2) \, dy = 0
\end{equation}
hence$ v(x,h)=0$. Further, by (ii) we get $v(x,0)=0$.
Differentiating (\ref{eq:(iii)'}) with respect to $x$ gives
\begin{equation}
\label{eq:D(iii)}
\int_{-h}^h w_y w_{xy} \, dy = 0. 
\end{equation}
Next we show that $v$ is odd in $y$: By assumption (ii) and equation \eqref{eq:(iii)'}
\begin{align*}
v(x,-y) & = \frac12 \int_{-h}^{-y} 1 - w_y^2(x,\tilde{y}) \, d \tilde{y} =
\frac12 \int_{y}^{h} 1 - w_y^2(x,\tilde{y}) \, d \tilde{y}
\\
& = \frac12 \int_{-h}^{h} 1 - w_y^2(x,\tilde{y}) \, d \tilde{y}  - \frac12 \int_{-h}^{y} 1 - w_y^2(x,\tilde{y}) \, d \tilde{y}
\\
& = - v(x,y)
\end{align*}
as required.

  Now we turn to the boundary conditions on $u$. Clearly $u = \frac 12 x$ for $y=-h$ by the definition of $u$. Note that, by (ii), $w_x w_y$ is an odd function of $y$. Therefore
 \begin{equation}
 \label{eq:wxwy}
 \int_{-h}^h w_x w_y \, d y=0.
 \end{equation}
Since $v$ is an odd function of $y$, so is $v_x$, and
\begin{equation}
\label{vx}
\int_{-h}^h v_x \, d y = 0.
\end{equation}
For future use we differentiate this equation with respect to $x$ and record that
\begin{equation}
\label{vxx}
\int_{-h}^h v_{xx} \, d y = 0.
\end{equation}
By equations \eqref{uv}$_2$, \eqref{eq:wxwy}, \eqref{vx}
\begin{equation*}
u(x,h) = \frac{1}{2} x - \int_{-h}^h (w_x w_y + v_x) \, d \tilde{y} = \frac{1}{2} x.
\end{equation*}
It remains to show that $u = \frac 12 x$ for $x=0,l$. This follows immediately from the definition of $u$ and $v$ and
the fact that $w_x$ is zero close to $x=0,l$ (assumption (i)).

Now we compute the energy bound \eqref{eq:EbB}. It is easy to see that the bonding energy
$I_{\textrm{Bo}} := \gamma | \{ (x,y) \in \Omega : w(x,y) > 0 \} |$
satisfies $I_{\textrm{Bo}} \le c \gamma \delta l$ in both cases (v) and (vi), which gives the
first term in equation \eqref{eq:EbB}.

To estimate the bending energy we observe that by (iv) and the assumption $h \le l$
  \begin{equation*}
    |D^2w|^2 = w_{xx}^2 + 2 w_{xy}^2 + w_{yy}^2 \le c h \delta \left( \left( \frac{h}{\delta l} \right)^4
    + 2 \left( \frac{h}{\delta l} \right)^2 \left( \frac{1}{\delta} \right)^2 + \left( \frac{1}{\delta} \right)^4 \right)
    \le c \frac{h \delta}{\delta^4} = c \frac{h}{\delta^3}.
  \end{equation*}
  Therefore in both cases (v) and (vi) the bending energy satisfies
  \begin{equation}
  (\sigma l_1)^2 \int_\Omega |D^2w|^2 \, d \bx \le c (\sigma l_1)^2 \delta l \frac{h}{\delta^3} = c (\sigma l_1)^2 \frac {lh}{\delta^2},
\end{equation}
which gives the second term in equation \eqref{eq:EbB}.

Now we come to the stretching energy. By the definition of $u$ and $v$
we only need to estimate
\begin{equation*}
I_{\textrm{S}} = \int_\Omega (2 u_x + w_x^2 - 1)^2 \, d \bx
\end{equation*}
(see Remark \ref{Remark:assump}).
From the definition of $u$, equation  \eqref{vxx}, and the fact that $w_y w_{xx}$ is an odd function of $y$ (assumption (ii)), we compute
  \begin{equation}
  \label{ux}
    u_x = \frac12 - \int_{-h}^y (w_{xy} w_x + w_y w_{xx} + v_{xx})  \, d \tilde{y}
    = \frac12 + \int_{y}^h (w_{xy} w_x + w_y w_{xx} + v_{xx} ) \, d \tilde{y}.
  \end{equation}
From the definition of $v$ and equation \eqref{eq:D(iii)} we see that
\begin{equation}
\label{v_x}
    v_x = - \int_{-h}^y w_y w_{xy} \, d \tilde{y} = \int_y^h w_y w_{xy} \, d \tilde{y}
  \end{equation}
  and therefore
    \begin{equation}
    \label{vxx2}
    v_{xx} = - \int_{-h}^y w_{xy}^2 + w_y w_{xxy} \, d \tilde{y} = \int_y^h w_{xy}^2 + w_y w_{xxy} \, d \tilde{y}.
  \end{equation}
By assumption (iv)
  \begin{equation}
  \label{wxywx+wywxx}
   | w_{xy} w_x + w_y w_{xx}|
    \le \left( \frac{h}{\delta l} \right) \left( \frac{1}{\delta} \right) (h \delta)^{1/2} \left( \frac{h}{\delta l} \right) (h \delta)^{1/2}
    + \left( \frac{1}{\delta} \right) (h \delta)^{1/2} \left( \frac{h}{\delta l} \right)^2 (h \delta)^{1/2}
    \le c \frac{h^3}{\delta^2l^2}.
  \end{equation}

First we consider case (v).
Since $w=0$  for $|y|\ge\delta$, then from
equations \eqref{ux} and \eqref{vxx2} we see that $u_x = \frac 12$ for $|y|\ge\delta$.
Therefore the stretching energy
reduces to
\begin{equation}
\label{ISreduced}
I_{\textrm{S}} = \int_{0}^{l} \int_{-\delta}^\delta (2 u_x + w_x^2 - 1)^2 \, d y dx.
\end{equation}
For $|y|<\delta$ equations \eqref{ux}, \eqref{vxx2} reduce to
\begin{gather}
\label{ux2}
u_x = \frac12 - \int_{-\delta}^y w_{xy} w_x + w_y w_{xx} + v_{xx}  \, d \tilde{y}
    = \frac12 + \int_{y}^\delta w_{xy} w_x + w_y w_{xx} + v_{xx}  \, d \tilde{y},
    \\
    \label{vxx2.5}
    v_{xx} = - \int_{-\delta}^y w_{xy}^2 + w_y w_{xxy} \, d \tilde{y} = \int_y^\delta w_{xy}^2 + w_y w_{xxy} \, d \tilde{y}.
\end{gather}
By equation \eqref{vxx2.5} and assumption (iv)
\begin{equation}
\label{vxx3}
    |v_{xx}|\le c \delta \left( \left( \frac{h}{\delta l} \right)^2 \left( \frac{1}{\delta} \right)^2 h \delta
    + \frac{(h \delta)^{1/2}}{\delta} \left( \left( \frac{h}{\delta l} \right)^2 \left( \frac{1}{\delta} \right) (h \delta)^{1/2} \right) \right) \le c \frac{h^{3}}{\delta^{2} l^2}.
  \end{equation}
Using \eqref{ux2}, \eqref{wxywx+wywxx}, \eqref{vxx3} we estimate
\begin{equation}
\label{2ux-1}
|2 u_x - 1|\le c \delta \left( \frac{h^3}{\delta^2l^2} + \frac{h^{3}}{\delta^{2} l^2} \right) \le c\frac{h^3}{\delta l^2}.
\end{equation}
We conclude from equation \eqref{ISreduced}, \eqref{2ux-1} and assumption (iv) that
\begin{equation}
\label{IS(v)}
I_{\textrm{S}} \le 2\int_{0}^{l} \int_{-\delta}^\delta |2 u_x - 1|^2+|w_x|^4 \, dy dx \le c l \delta\frac{h^6}{\delta^2l^4}
= c\frac{h^6}{\delta l^3}.
\end{equation}
This gives the third term in equation \eqref{eq:EbB} and completes the proof for case (v).

We finish by computing the stretching energy for case (vi).
We decompose $\Omega$ as
$\Omega = \Omega_1 \cup \Omega_2 \cup \Omega_3$
where
\begin{align*}
\Omega_1 & := \{ (x,y) \in \Omega : \delta-\varphi(x) < y < - \delta + \varphi(x) \},
\\
\Omega_2 & := \{ (x,y) \in \Omega : -\delta + \varphi(x) \le y \le \delta + \varphi(x) \} \cup \{ (x,y) \in \Omega : -\delta - \varphi(x) \le y \le \delta - \varphi(x) \},
\\
\Omega_3 & := \{ (x,y) \in \Omega : y > \delta + \varphi(x) \} \cup \{ (x,y) \in \Omega : y < -\delta - \varphi(x) \}.
\end{align*}
Observe that $w=0$ in $\Omega_1$ and $\Omega_3$. $\Omega_1$ is the central region between the two folds, which occupy region $\Omega_2$. The same arguments used in the proof of case (v) show that the stretching energy in region $\Omega_3$ is zero and that the stretching energy in region $\Omega_2$ satisfies estimate \eqref{IS(v)}. (The integral over  $(-\delta,\delta)$ in \eqref{IS(v)} is replaced with integrals over $(-\delta-\varphi(x),\delta-\varphi(x))$ and $(-\delta+\varphi(x),\delta+\varphi(x))$.)
It remains to compute the stretching energy in the central region $\Omega_1$. Since $w=0$ in this region we only need to estimate
\begin{equation*}
\int_{\Omega_1} |2 u_x - 1|^2 \, dx dy.
\end{equation*}
First we show that $v_x=0$ in $\Omega_1$.
We only need to consider values of $x$ such that $\varphi(x)>\delta$.
Since $w=0$ in $\Omega_1$ and $\Omega_3$, equation \eqref{v_x} reduces to the following in $\Omega_1$:
\begin{equation}
\label{v_x2}
    v_x = - \int_{-\varphi(x)-\delta}^{-\varphi(x)+\delta} w_y w_{xy} \, d\tilde{y} = \int_{\varphi(x)-\delta}^{\varphi(x)+\delta} w_y w_{xy} \, d\tilde{y}
    \qquad \textrm{in } \Omega_1.
  \end{equation}
  But $w_y w_{xy}$ is even in $y$. Therefore
\begin{equation}
\label{v_x3}
\int_{\varphi(x)-\delta}^{\varphi(x)+\delta} w_y w_{xy} \, d\tilde{y}
=
\int_{-\varphi(x)-\delta}^{-\varphi(x)+\delta} w_y w_{xy} \, d\tilde{y}.
\end{equation}
Combining \eqref{v_x2} and \eqref{v_x3} proves that $v_x=0$ in $\Omega_1$, as claimed. Recalling (\ref{uv}), it follows that
\begin{equation*}
u_y = - w_x w_y - v_x = 0 \qquad \textrm{in } \Omega_1.
\end{equation*}
Therefore $u(x,y)=u(x,0)$ for all $(x,y) \in \Omega_1$. Using assumption (vi) we compute, for $y<0$,
\begin{align*}
w_x & =
 A'(x)\psi_\delta(y + \varphi(x))  + A(x) \varphi'(x)\psi'_\delta(y + \varphi(x)),
\\
w_y & = A(x) \psi'_\delta(y + \varphi(x))
\end{align*}
(since $(x,y)\in \Omega_1$, we have $\varphi(x)\ge \delta$).
It follows that, again for $y<0$,
\begin{equation*}
w_x w_y = A(x) A'(x) \frac12 \frac{d}{dy}  \psi_\delta^2(y + \varphi(x)) +  w_y^2(x,y) \varphi'(x).
\end{equation*}
Integrating gives
\begin{align*}
 \int_{-h}^0 w_x w_y \, dy =
& = \frac12 A(x) A'(x)\{\psi_\delta^2(\varphi(x))
-\psi_\delta^2(-h + \varphi(x))
\}
 + \varphi'(x) \int_{-h}^0  w_y^2 \, dy
 \\
 & = \frac12 A(x) A'(x) \psi_\delta^2(\varphi(x)) + \varphi'(x) h
\end{align*}
by assumption (iii) and since 
$\varphi(x) \in [0,h/2]$, $\delta < h/2$, and $\psi_\delta$ is supported on $[-\delta,\delta]$.
Therefore
\begin{equation*}
2u(x,0)-x  = - 2 \int_{-h}^0 (w_x w_y + v_x) \, dy
 = - A(x) A'(x) \psi_\delta^2(\varphi(x)) -2 h \varphi'(x) - 2 \int_{-h}^0 v_x \, dy.
\end{equation*}
If $(x,0) \in \Omega_1$, then $\varphi(x) > \delta$ and so $\psi_\delta(\varphi(x))=0$. Hence
\begin{equation*}
2 u(x,y)-x = 2u(x,0)-x = -2 h \varphi'(x) - 2 \frac{\partial}{\partial x} \int_{-h}^0 v \, dy \quad \textrm{ for all } (x,y) \in \Omega_1.
\end{equation*}
We compute the integral on the right-hand side:
\begin{equation*}
- 2 \frac{\partial}{\partial x}  \int_{-h}^0 v \, d y  =
\frac{\partial}{\partial x}  \int_{-h}^0 \int_{-h}^y w_y^2(x,\tilde{y})  \, d \tilde{y}  \, dy
 = \frac{\partial}{\partial x} \int_{-h}^0 \int_{\tilde{y}}^0 w_y^2(x,\tilde{y}) \, d y \, d\tilde{y}
\end{equation*}
by changing the order of integration. Therefore
\begin{equation*}
- 2 \frac{\partial}{\partial x}  \int_{-h}^0 v \, d y  =
- \frac{\partial}{\partial x} \int_{-h}^0 \tilde{y} w_y^2(x,\tilde{y}) \, d\tilde{y}
= \frac{\partial}{\partial x} \int_{0}^h y w_y^2(x,y) \, dy.
\end{equation*}
If $(x,0) \in \Omega_1$, then
\begin{align*}
- 2 \frac{\partial}{\partial x}  \int_{-h}^0 v \, d y
& = \frac{\partial}{\partial x} \int_{\varphi(x)-\delta}^{\varphi(x)+\delta} y w_y^2(x,y) \, dy
\\
& = \frac{\partial}{\partial x} \int_{\varphi(x)-\delta}^{\varphi(x)+\delta} (y-\varphi(x)) w_y^2(x,y) \, dy + h \varphi'(x) \qquad \textrm{(by assumption (iii))}
\\
& = \frac{\partial}{\partial x} \int_{\varphi(x)-\delta}^{\varphi(x)+\delta} (y-\varphi(x)) A^2(x) [\psi'_\delta(y - \varphi(x))]^2 \, dy  + h\varphi'(x)
\\
& = \frac{\partial}{\partial x} \int_{-\delta}^{\delta} y A^2(x) [\psi'_\delta(y)]^2 \, dy  + h \varphi'(x)
\qquad \textrm{(by changing variables)}
\\
& = 2 A(x) A'(x) \int_{-\delta}^{\delta} y [\psi'_\delta(y)]^2 \, dy + h \varphi'(x)
\\
& = h \varphi'(x)
\end{align*}
since $y [\psi'_\delta(y)]^2$ is odd. Therefore, if $(x,y) \in \Omega_1$,
\begin{equation*}
2 u(x,y)-x = - h \varphi'(x).
\end{equation*}
We conclude that
\begin{equation*}
\int_{\Omega_1} |2 u_x - 1|^2 \, dx dy = \int_{\Omega_1} h^2 |\varphi''(x)|^2 \, dx dy
\le C h l \, h^2 (h/l^2)^2 = C \frac{h^5}{l^3} \le C \frac{h^6}{\delta l^3}
\end{equation*}
since $\delta < h$.
Therefore the stretching energy in region $\Omega_1$ is the same order (or less) than the stretching energy in region $\Omega_2$ and we have finally arrived at the desired estimate \eqref{eq:EbB}.
\end{proof}

%
%
%
In the follow lemma we construct a deformation that takes one fold of width $2\delta$ at $x=l$ and splits it into two folds of width $2\delta$ separated by a distance of $h-2 \delta$ at $x=0$. The film is bonded to the substrate between the two folds. See Figure \ref{RegimeE}, left.
\begin{Lemma}[Fold splitting construction]
\label{LemmaFSS}
Let $0 < \delta \le h/2$.
Let $\psi\in C^\infty_c((-1,1);[0,1])$ be an even function with $\psi(0)=1$. Define $\psi_\delta(y)=\psi(y/\delta)$.
Let $\varphi\in C^\infty([0,l];[0,h/2])$ satisfy $\varphi(0)=h/2$, $\varphi(l)=0$,
$\varphi'=\varphi''=0$ in a neighbourhood of $x=0,l$,
and
$|\varphi^{(k)}|\le C h/l^k$ for $k\in\{0,1,2\}$.
Set
\begin{equation*}
\tilde{w}(x,y):=\psi_\delta(y+\varphi(x))+\psi_\delta(y-\varphi(x)).
\end{equation*}
Define $A:[0,l]\to(0,\infty)$ by
\begin{equation}\label{eqdefA}
A^2(x)  \int_{0}^h \tilde{w}_y^2(x,y) \, dy =h.
\end{equation}
Then $w:[0,l]\times[-h,h] \to [0,\infty)$ defined by
\begin{equation}
w(x,y):=A(x)\tilde{w}(x,y)
\end{equation}
satisfies the assumptions of Lemma \ref{LemmaUB}.
\end{Lemma}
\begin{proof}
Observe that $w$ satisfies assumption (ii) of Lemma \ref{LemmaUB} since
$\psi_\delta$ is even.
Equation \eqref{eqdefA} ensures that $w$ satisfies assumption (iii). Clearly $w$ also satisfies assumption (vi).
It remains to check assumptions (i) and (iv).

Since $\delta \le h/2$ and $|\varphi| \le h/2$, then $w(x,y)=0$ for $y$ in a neighbourhood of $\pm h$.
We differentiate (\ref{eqdefA}) with respect to $x$ and rearrange to get
\begin{equation}
\label{A'}
A'(x) = \frac{\displaystyle - A (x)\int_{0}^h  \tilde{w}_y \tilde{w}_{xy} \, dy}{\displaystyle \int_{0}^h \tilde{w}_y^2 \, dy}.
\end{equation}
By assumption $\varphi'$ vanishes in a neighbourhood of $x=0,l$, and consequently so do $\tilde{w}_{xy}$, $A'$ and
\begin{equation*}
w_x(x,y) = A'(x) \tilde{w}(x,y) + A(x)(\psi'_\delta(y+\varphi(x))+\psi'_\delta(y-\varphi(x))) \varphi'(x).
\end{equation*}
Therefore assumption (i) of Lemma \ref{LemmaUB} is satisfied.

Finally we check assumption (iv). Observe that
\begin{equation*}
\int_{0}^h \tilde{w}_y^2(x,y) \, dy
\ge
\int_{\varphi(x)}^{\varphi(x)+\delta} \tilde{w}_y^2(x,y) \, dy
\ge \frac1\delta \left(
\int_{\varphi(x)}^{\varphi(x)+\delta} \tilde{w}_y(x,y) \, dy\right)^2
=\frac1\delta \tilde w^2(x,\varphi(x))\ge \frac1\delta
\end{equation*}
since $\psi(0)=1$. Observing that by definition
$|\tilde w_y|\le C/\delta$
we conclude
\begin{equation}
\label{intBound}
\frac{1}{\delta} \le \int_{0}^h \tilde{w}_y^2 \, dy \le \frac{C}{\delta} \qquad \textrm{ and } \qquad c (h\delta)^{1/2} \le A\le C (h\delta)^{1/2}.
\end{equation}
Since $|\tilde{w}_y| \le C \delta^{-1}$ and $|\tilde{w}_{xy}| \le C | \varphi'| \delta^{-2}$,  then \eqref{A'} and \eqref{intBound} give
\begin{equation*}
  | A'| = |A| \left| \int_{\max \{0,-\delta+\varphi(x)\} }^{\delta+\varphi(x)} \tilde{w}_y \tilde{w}_{xy} \, dy \right| \;
  \left| \int_0^h \tilde{w}_y^2 \, dy \right|^{-1}   \le  C \frac{|A| |\varphi'|}{\delta} \le C |A| \frac{h}{\delta l}.
\end{equation*}
Note that
\begin{align*}
| \tilde{w}_{xxy} | & \le |\psi'''_\delta(y+\varphi(x))+\psi'''_\delta(y-\varphi(x)) | | \varphi' |^2 +
|\psi''_\delta(y+\varphi(x))+\psi''_\delta(y-\varphi(x)) | | \varphi'' |
\\
& \le C \frac{1}{\delta^3} \left( \frac{h}{l} \right)^2 + C \frac{1}{\delta^2} \frac{h}{l^2} \le C \frac{h^2}{\delta^3 l^2}
\end{align*}
since $\delta < h$.
Therefore
\begin{align*}
|A''| & \le \frac{\displaystyle \left| -A' \int_{0}^h  \tilde{w}_y \tilde{w}_{xy} \, dy - A
\int_{0}^h  ( \tilde{w}_{xy}^2 + \tilde{w}_y \tilde{w}_{xxy} ) \, dy
\right|}{\displaystyle \int_{0}^h \tilde{w}_y^2 \, dy}
+2|A| \frac{\displaystyle \left( \int_{0}^h  \tilde{w}_y \tilde{w}_{xy} \, dy \right)^2 }{\displaystyle \left(\int_{0}^h \tilde{w}_y^2 \, dy\right)^2}
\\
& \le C \delta \left( |A'| \delta \delta^{-1} |\varphi'| \delta^{-2} + |A| \delta (|\varphi'|^2 \delta^{-4} + \delta^{-1}h^2 \delta^{-3} l^{-2})
\right) + C |A| \delta^2 \left( \delta \delta^{-1} |\varphi'| \delta^{-2} \right)^2
\\
& \le C |A| \left( \frac{h}{\delta l} \right)^2
\end{align*}
since $|A'| \le C |A| |\varphi'|/ \delta$ and $|\varphi'| \le C h/l$.
Putting everything together gives
\begin{align*}
|w| & \le |A| |\tilde{w}| \le C (h \delta)^{1/2}, \\
|w_x| & \le |A'| |\tilde{w}| + |A| |\psi'_\delta(y+\varphi(x))+\psi'_\delta(y-\varphi(x)) | |\varphi'| \le C |A| |\varphi'| \delta^{-1}
\le C (h \delta)^{1/2} \frac{h}{\delta l}, \\
|w_y| & \le |A| |\psi'_\delta(y+\varphi(x))+\psi'_\delta(y-\varphi(x)) | \le C (h \delta)^{1/2} \frac{1}{\delta},
\\
|w_{xx} | & \le |A''| |\tilde{w}| + 2 |A'| |\psi'_\delta(y+\varphi(x))+\psi'_\delta(y-\varphi(x)) | |\varphi'|
+|A| |\psi''_\delta(y+\varphi(x))+\psi''_\delta(y-\varphi(x)) | |\varphi'|^2
\\ & \quad
+ |A| |\psi'_\delta(y+\varphi(x))+\psi'_\delta(y-\varphi(x)) | |\varphi''|
\\
& \le C (h \delta)^{1/2} \left( \frac{h^2}{\delta^2 l^2} +  \frac{h}{\delta l^2} \right)
\le C (h \delta)^{1/2} \left( \frac{h}{\delta l} \right)^2,
\\
|w_{xy} | & \le |A'| |\psi'_\delta(y+\varphi(x))+\psi'_\delta(y-\varphi(x)) | + |A| |\psi''_\delta(y+\varphi(x))+\psi''_\delta(y-\varphi(x)) | |\varphi'|
\\
& \le C  (h \delta)^{1/2} \left( \frac{h}{\delta l} \right) \left( \frac{1}{\delta} \right),
\\
|w_{yy} | & \le |A| |\psi''_\delta(y+\varphi(x))+\psi''_\delta(y-\varphi(x)) | \le C (h \delta)^{1/2} \left( \frac{1}{\delta} \right)^2
\end{align*}
and hence $w$
satisfies assumption (iv) of Lemma \ref{LemmaUB}.
\end{proof}
In the following lemma we construct a deformation that takes one fold of width $2 \delta$ at $x=l$ and shrinks it down to a fold of width $2 \tilde{\delta} < 2 \delta$ at $x=0$. See Figure \ref{RegimeE}, right.
\begin{Lemma}[Fold shrinkage construction]
\label{Lemm:FS}
Let $0 < \delta \le h$, $\lambda\in [1/4,1]$.
Let $\psi\in C^\infty_c((-1,1);[0,1])$ be an even function with $\psi(0)=1$.  Define $\psi_\delta(y):=\psi(y/\delta)$.
Let $\hat{\delta} \in C^2([0,l];[\lambda\delta,\delta])$ satisfy $\hat{\delta}(0)=\lambda{\delta}$, $\hat{\delta}(l)=\delta$,
$\hat{\delta}'=0$ in a neighbourhood of $x=0,l$,
and $|\hat{\delta}^{(k)}| \le C \delta/l^k$. Set
$\tilde{w}(x,y):=\psi_{\hat{\delta}(x)}(y)$.
Define $A:[0,l]\to(0,\infty)$ by
\begin{equation}\label{eqdefA2}
A^2(x)  \int_{0}^h \tilde{w}_y^2(x,y) \, dy =h.
\end{equation}
Then $w:[0,l]\times[-h,h] \to [0,\infty)$ defined by
\begin{equation}
w(x,y):=A(x)\tilde{w}(x,y)
\end{equation}
satisfies the assumptions of Lemma \ref{LemmaUB}.
\end{Lemma}
\begin{proof}
The same arguments used in the proof of Lemma \ref{LemmaFSS} show that $w$ satisfies assumptions (i)--(iii) of Lemma \ref{LemmaUB}.
Clearly it also satisfies assumption (v). It remains to check assumption (iv).
Observe that
\begin{equation*}
\int_0^h \tilde{w}_y^2(x,y) \, dy
= \int_0^h [\psi_{\hat{\delta}(x)}'(y)]^2 \, dy
= \frac{1}{\hat{\delta}^2(x)} \int_0^{\hat{\delta}(x)} \left[ \psi' \left(\tfrac{y}{\hat{\delta}(x)} \right) \right]^2 \, dy
= \frac{1}{\hat{\delta}(x)} \int_0^1 [\psi'(y)]^2 \, dy = \frac{c_*}{\hat{\delta}(x)}.
\end{equation*}
Therefore
\begin{equation*}
A(x) = \left(\frac{\hat{\delta}(x) h}{c_*}\right)^{1/2} \le c (\delta h)^{1/2}.
\end{equation*}
Note that
\begin{equation*}
\frac{\partial}{\partial \delta} \psi_\delta (y) = - \psi' \left( \frac{y}{\delta} \right) \frac{y}{\delta^2}.
\end{equation*}
Since $\hat{\delta}(x) \ge \lambda \delta$, we estimate
\begin{align*}
|\tilde{w}_x(x,y)| & = \left| \psi' \left( \frac{y}{\hat{\delta}(x)} \right)  \frac{y}{\hat{\delta}^2(x)} \hat{\delta}'(x) \right| \le C
\left| \frac{\hat{\delta}'(x)}{\hat{\delta}(x)} \right| \le \frac{C}{l},
\\
|\tilde{w}_y(x,y)| & = \left| \psi' \left( \frac{y}{\hat{\delta}(x)} \right)  \frac{1}{\hat{\delta}(x)} \right| \le \frac{C}{\delta},
\\
 |\tilde{w}_{yx}(x,y)| & \le \left| \psi'' \left( \frac{y}{\hat{\delta}(x)} \right)  \frac{y}{\hat{\delta}^3(x)} \hat{\delta}'(x) \right|
 + \left| \psi' \left( \frac{y}{\hat{\delta}(x)} \right)  \frac{\hat{\delta}'(x)}{\hat{\delta}^2(x)}  \right|
 \le C \left| \frac{\hat{\delta}'(x)}{\hat{\delta}^2(x)} \right| \le \frac{C}{\delta l},
 \\
 |\tilde{w}_{xxy}(x,y)| & \le
 \left| \psi''' \left( \frac{y}{\hat{\delta}(x)} \right)  \frac{y^2}{\hat{\delta}^5(x)} (\hat{\delta}'(x))^2 \right|
 +
3 \left| \psi'' \left( \frac{y}{\hat{\delta}(x)} \right)  \frac{y}{\hat{\delta}^4(x)} (\hat{\delta}'(x))^2 \right|
\\
& \; \,
+
\left| \psi'' \left( \frac{y}{\hat{\delta}(x)} \right)  \frac{y}{\hat{\delta}^3(x)} \hat{\delta}''(x) \right|
+
\left| \psi' \left( \frac{y}{\hat{\delta}(x)} \right)  \frac{1}{\hat{\delta}^2(x)} \hat{\delta}''(x) \right|
+
2 \left| \psi' \left( \frac{y}{\hat{\delta}(x)} \right)  \frac{1}{\hat{\delta}^3(x)} (\hat{\delta}'(x))^2 \right|
\\
& \le C \frac{1}{\delta l^2}.
\end{align*}
Similarly to the proof of Lemma \ref{LemmaFSS}, it follows that
\begin{equation*}
| A' | \le C \frac{|A|}{l}, \qquad |A''| \le C \frac{|A|}{l^2}.
\end{equation*}
The estimates on the derivatives of $\tilde{w}$ and $A$ are at least as good those given in the proof of Lemma \ref{LemmaFSS}
(in fact they are better). Therefore the derivatives of $w$ satisfy the same estimates proved in Lemma \ref{LemmaFSS}, and
$w$ satisfies assumption (iv) of Lemma \ref{LemmaUB}, as required.
\end{proof}

\begin{remark}
From the proof of Lemma \ref{Lemm:FS} we see that fold shrinkage is cheaper than fold splitting.
\end{remark}

\subsection{Overall construction.}
In this section we put everything together to prove Theorem \ref{UB1} for the branching regimes $C$ and $D$.
\begin{Proposition}[Overall construction for $0\le \gamma\le \sigma^{-4/9}$]
\label{Prop:overall}
  Let $0<l_1\le l_2$ and $ \gamma\le \sigma^{-4/9}$. Then
  \begin{equation*}
    \min_V I^{(\sigma,\gamma)} \le C l_1 l_2 (\sigma^{1/2}\gamma^{5/8} + \sigma).
  \end{equation*}
\end{Proposition}

\begin{remark}
\label{UB3}
Observe that
\begin{equation*}
\sigma^{1/2} \g^{5/8} + \sigma \le
\left\{
\begin{array}{cl}
2 \sigma^{1/2} \g^{5/8} & \textrm{if } \, \g \ge \sigma^{4/5}, \; \textrm{in particular if} \; (\sigma,\gamma) \in C,  \\
2 \sigma & \textrm{if } \, \g < \sigma^{4/5}, \; \textrm{i.e., if} \; (\sigma,\gamma) \in D.
\end{array}
\right.
\end{equation*}
Also
\begin{equation*}
\sigma^{1/2} \g^{5/8} + \sigma \ge (\sigma \gamma)^{2/5} \; \textrm{ if } \; \g \ge \sigma^{-4/9}, \; \textrm{in particular if} \; (\sigma,\gamma) \in B.
\end{equation*}
\end{remark}
\begin{figure}
\centerline{\includegraphics[height=0.3\textheight]{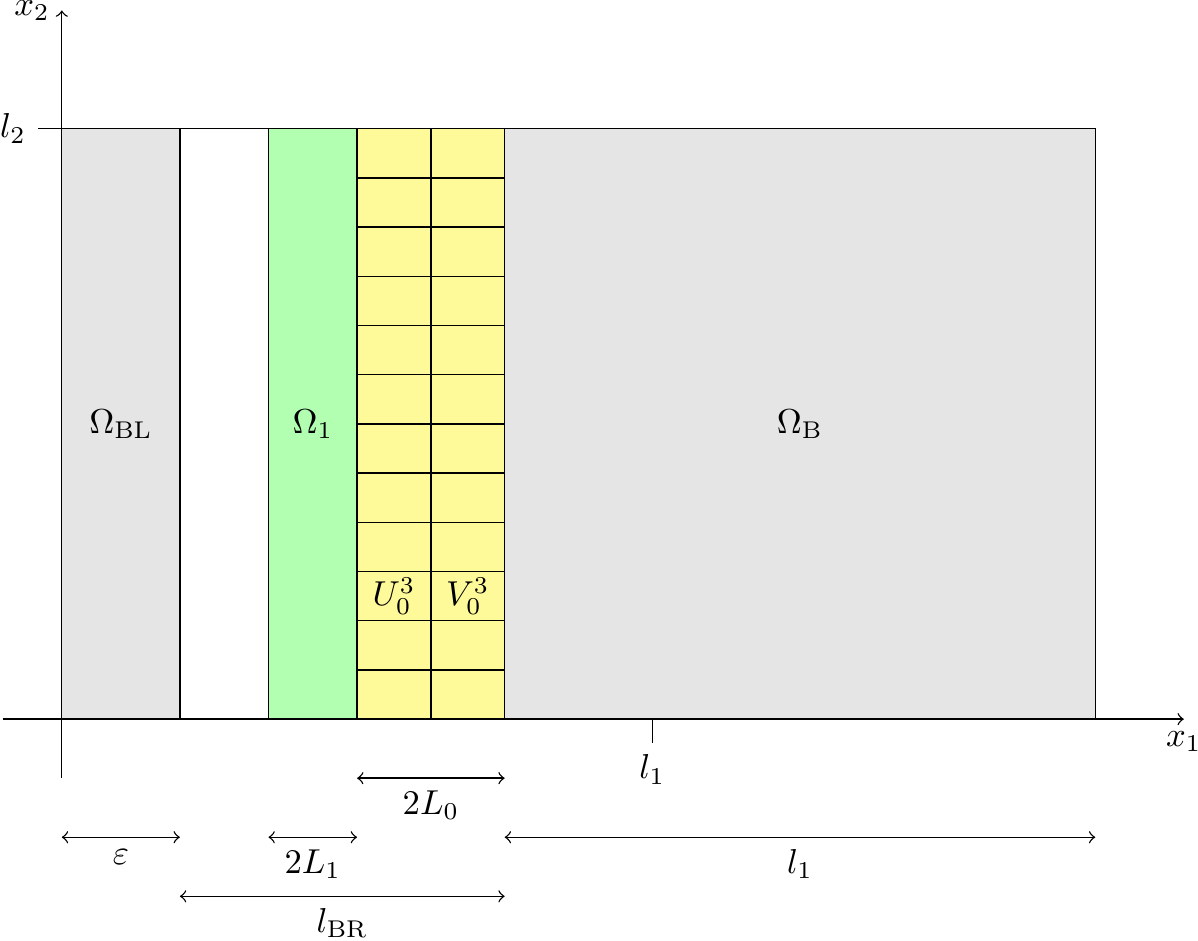}}
\caption{Sketch of the geometry in the proof of Proposition \ref{Prop:overall}. The sets are so defined that $\Omega$ is covered by
their union, this is ensured by making $\Omega_\textrm{B}$ have width $l_1$.}
\label{Fig-ub1}
\end{figure}

\begin{proof}[Proof of Proposition \ref{Prop:overall}]
We cover  $\Omega = (0,l_1) \times (0,l_2)$ by the union of three sets: a boundary layer region $\Omega_{\textrm{BL}}$, a branching region $\Omega_{\textrm{BR}}$, and a bulk region
$\Omega_{\textrm{B}}$,
\begin{equation*}
\Omega \subset \tilde\Omega:=\Omega_{\textrm{BL}} \cup \Omega_{\textrm{BR}} \cup \Omega_{\textrm{B}}
\end{equation*}
where
\begin{equation*}
\Omega_{\textrm{BL}} := (0,\varepsilon) \times (0,l_2), \quad \Omega_{\textrm{BR}} := [\varepsilon,\varepsilon+l_\textrm{BR}] \times (0,l_2), \quad
\Omega_{\textrm{B}} := (\varepsilon+l_\textrm{BR},\varepsilon+l_\textrm{BR}+l_1) \times (0,l_2),
\end{equation*}
and where the boundary layer width $\varepsilon$ and the branching  region width $l_\textrm{BR}$ are to be determined.
In particular,  $l_\textrm{BR}$  will be chosen by summing up the optimal widths of the individual branching steps.
We construct the test function on $\tilde\Omega$ and then restrict to $\Omega$,
which only decreases the energy; 
note that $\Omega \subset \tilde\Omega$ for any choice of the parameters.
In the
boundary layer region $\Omega_{\textrm{BL}}$ we use the construction $(\tilde{u},\tilde{v},\tilde{w})$ from Section
\ref{Bs2} with $h = h_N$ and $\delta = \delta_N$ to be determined. In the
bulk region $\Omega_{\textrm{B}}$ we use the construction from Section
\ref{Bs1} with $l = l_1$, $h = h_0$ and $\delta = \delta_0$, again to be determined.
We assume $h_N=2^Nh_0$ for some $N\in\N$, also to be determined.
(We assume without loss of generality as in Section \ref{Bs3} that
$2 h_0$ divides $l_2$ and use the construction from Section \ref{Bs1} on the $l_2/2h_0$ rectangles $ (\varepsilon+l_\textrm{BR},\varepsilon+l_\textrm{BR}+l_1)\times (2h_0(j-1),2h_0j)$, $j \in \{1 , \ldots , l_2/2 h_0\}$.)

We split up the branching region $\Omega_{\textrm{BR}}$ into $N$ vertical strips $\Omega_i$ of width $2 L_i$, $i \in \{0,\ldots,N-1\}$:
\begin{equation*}
\Omega_{\textrm{BR}} = \bigcup_{i=0}^{N-1} \Omega_i, \qquad \Omega_i :=
\Bigg\{ [0,2L_i] + \varepsilon + \sum_{k=i+1}^{N-1} 2L_k \Bigg\} \times (0,l_2).
\end{equation*}
Note that the strips are labelled from left to right in decreasing order $\Omega_{N-1}, \ldots , \Omega_1, \Omega_0$.
The total width of the branching region is defined by 
$l_\textrm{BR}:=\sum_{i=0}^{N-1} 2L_i$, while the values of the $L_i$'s are chosen below.
Let $h_{i+1}=h_i/2$, $i \in \{0, \ldots , N-1 \}$.
Then each strip is split into $l_2/2h_i$ rectangles $R_i^j$ of height $2 h_i$, see Figure \ref{Fig-ub1}.
\begin{equation*}
\Omega_i = \bigcup_{j=1}^{l_2/2h_i} R_i^j, \qquad R_i^j := \Bigg\{ [0,2L_i] + \varepsilon + \sum_{k=i+1}^{N-1} 2L_k \Bigg\} \times
 \Bigg\{ [0,2 h_i] + (j-1)2h_i \Bigg\}.
\end{equation*}
Finally, we split each rectangle $R_i^j$ into two rectangles $U_i^j$, $V_i^j$ of width $L_i$:
\begin{equation*}
R_i^j = U_i^j \cup V_i^j
\end{equation*}
\begin{align*}
U_i^j & := \Bigg\{ [0,L_i] + \varepsilon + \sum_{k=i+1}^{N-1} 2L_k \Bigg\} \times \Bigg\{ [0,2 h_i] + (j-1)2h_i \Bigg\},
\\
V_i^j & := \Bigg\{ [L_i,2L_i] + \varepsilon + \sum_{k=i+1}^{N-1} 2L_k \Bigg\} \times \Bigg\{ [0,2 h_i] + (j-1)2h_i \Bigg\}.
\end{align*}
\begin{figure}
\centerline{\includegraphics[width=0.8\textwidth]{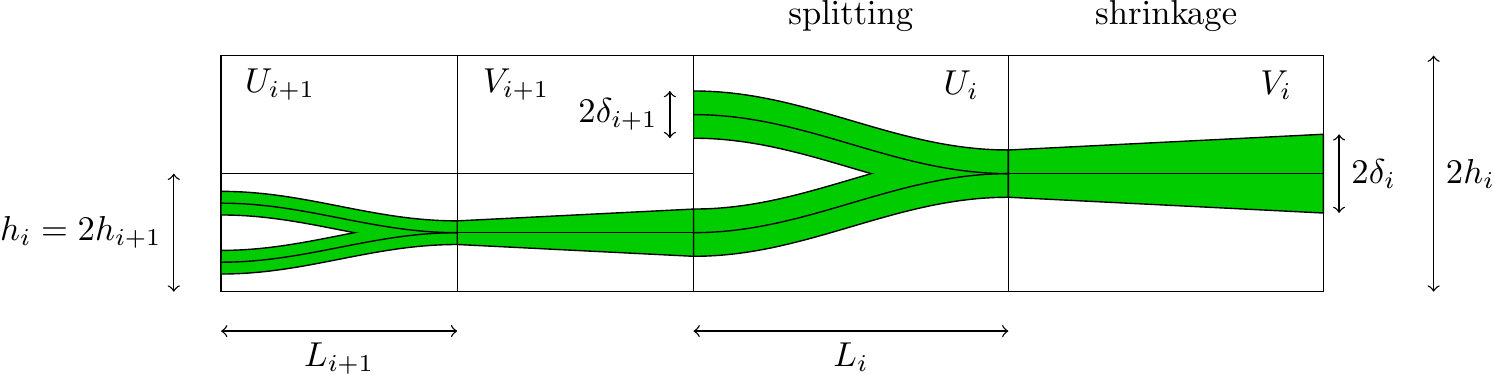}}
\caption{Sketch of the geometry in the proof of Proposition \ref{Prop:overall}. The area marked green is the set where $w>0$.}
\label{Fig-ub2}
\end{figure}
Let $\delta_0 > \delta_1 > \cdots > \delta_N$.
On the rectangles $U_i^j$ we use the fold splitting construction from Lemma \ref{LemmaFSS} with
$h = h_i$, $\delta = \delta_{i+1}$, $l=L_i$. This is admissible provided that
\begin{equation}
 \label{eqadmonUij}
 \delta_{i+1}\le h_{i+1}=\frac{h_i}2 \text{ and } h_i\le L_i\,.
\end{equation}
On the rectangles $V_i^j$ we use the fold shrinkage construction from Lemma \ref{Lemm:FS} with
$h = h_i$, $\delta = \delta_{i}$, $\lambda=\delta_{i+1}/\delta_{i}$, $l=L_i$, see Figure \ref{Fig-ub2}.
This is admissible provided that
\begin{equation}
 \label{eqadmonVij}
 \delta_{i}\le h_{i}\le L_i \text{ and } \frac14\delta_i\le \delta_{i+1}\le \delta_i\,.
\end{equation}
By Lemmas \ref{LemmaUB},   \ref{LemmaFSS} and \ref{Lemm:FS} the energy of the displacement satisfies
\begin{equation}
\label{eq:ebound}
I^{(\sigma,\gamma)}[\bu,w,U_i^j\cup V_i^j]\le  c \gamma \delta_i L_i + c (\sigma l_1)^2 \frac{L_i h_i }{\delta_i^2}
  + c \frac{h_i^6}{\delta_i L_i^3}.
\end{equation}
Equating terms on the right-hand leads to the optimality conditions
\begin{equation*}
\delta_i \sim  (\sigma l_1)^{2/3} \gamma^{-1/3} h_i^{1/3}, \quad L_i \sim
(\sigma l_1)^{-1/2} \delta_i^{1/4} h_i^{5/4}.
\end{equation*}
In order to enforce (\ref{eqadmonUij}) and (\ref{eqadmonVij}) we choose
\begin{equation*}
\delta_i :=  \min \left\{ h_i , (\sigma l_1)^{2/3} \gamma^{-1/3} h_i^{1/3} \right\}, \quad
L_i :=
\max \left\{ h_i , (\sigma l_1)^{-1/2} \delta_i^{1/4} h_i^{5/4} \right\}, \quad i \in \{ 0, \ldots , N-1 \}.
\end{equation*}
We only have to verify the second condition in (\ref{eqadmonVij}). Obviously $\delta_{i+1}\le \delta_i$. At the same time,
$h_{i+1} =  h_i/2$ implies $\delta_{i+1} \ge \delta_i/2$.

To complete the proof we need to consider two cases: $\gamma \ge 1$, $\gamma < 1$.

\emph{Case $\gamma \ge 1$:} First we choose $h_N$ and $\delta_N$, the size of the folding pattern in the boundary layer, by equating terms in equation \eqref{B.16}. This gives
\begin{equation}
\label{hN:g>1}
\varepsilon := h_N := \sigma l_1 \gamma, \quad \delta_N :=    \sigma l_1
\end{equation}
and
\begin{equation}
\label{eq:UBBL}
I^{(\sigma,\gamma)}[\bu,w,\Omega_{\textrm{BL}}] \le C l_2 \varepsilon = C l_1 l_2 \sigma \gamma.
\end{equation}
Observe that $\delta_N \le \varepsilon$ since $\gamma \ge 1$ and so the bound \eqref{B.16} is valid
(in particular, $h\le l_2$ since $\sigma\gamma\le 1$).
Now we determine which of the two values is taken by $\delta_i$, $i \in \{ 0 , \ldots , N-1 \}$. Note that
\begin{equation*}
h_i  \ge (\sigma l_1)^{2/3} \gamma^{-1/3} h_i^{1/3} \quad \Longleftrightarrow \quad h_i \ge \sigma l_1   \gamma^{-1/2}.
\end{equation*}
But for all $i$
\begin{equation*}
h_i \ge h_N =\sigma  l_1  \gamma \ge  \sigma  l_1 \gamma^{-1/2}
 \end{equation*}
 since $\gamma \ge 1$. Therefore  for all $ i \in \{ 0 , \ldots , N-1 \}$
\begin{equation}
\label{eq:di}
\delta_i =  (\sigma l_1)^{2/3} \gamma^{-1/3} h_i^{1/3}
\end{equation}
and, inserting in the definition of $L_i$,
\begin{equation*}
L_i = \max \left\{ h_i , (\sigma l_1)^{-1/3} \gamma^{-1/12} h_i^{4/3} \right\}.
\end{equation*}
Note that
\begin{equation*}
h_i \le (\sigma l_1)^{-1/3} \gamma^{-1/12} h_i^{4/3} \quad \Longleftrightarrow \quad h_i \ge  \sigma l_1\gamma^{1/4}.
\end{equation*}
But for all $i$
\begin{equation*}
h_i \ge h_N =  \sigma  l_1\gamma \ge  \sigma l_1\gamma^{1/4}
 \end{equation*}
 since $\gamma \ge 1$. Therefore
 \begin{equation}
 \label{eq:Li}
L_i =  (\sigma l_1)^{-1/3} \gamma^{-1/12} h_i^{4/3} \quad \textrm{for all } i \in \{ 0 , \ldots , N-1 \}.
\end{equation}
It remains to choose $h_0$ (equivalently $N$). We do this by balancing the energy in the bulk region $\Omega_{\textrm{B}}$ with the energy in the
branching region $\Omega_{\textrm{BR}}$. The same computation as in  Lemma \ref{Energy:Bint} shows that the energy in $\Omega_{\textrm{B}}$ satisfies
\begin{equation}
\label{eq:UBB}
I^{(\sigma,\gamma)}[\bu,w,\Omega_{\textrm{B}}] \le C \frac{l_2}{2 h_0} \left(  (\sigma l_1)^2 \frac{l_1h_0}{\delta_0^2} + 2 \g l_1 \delta_0 \right)
= C l_1 l_2 (\sigma l_1)^{2/3} \gamma^{2/3} h_0^{-2/3}.
\end{equation}
Inserting  \eqref{eq:di} and \eqref{eq:Li} in (\ref{eq:ebound}) leads to
\begin{equation*}
I^{(\sigma,\gamma)}[\bu,w,U_i^j\cup V_i^j]\le
 C (\sigma l_1)^{1/3} \gamma^{7/12} h_i^{5/3}.
\end{equation*}
Therefore the energy in $\Omega_{\textrm{BR}}$ satisfies
\begin{align}
\nonumber
I^{(\sigma,\gamma)}[\bu,w,\Omega_{\textrm{BR}}] & \le C \sum_{i=0}^{N-1} \frac{l_2}{2h_i} (\sigma l_1)^{1/3} \gamma^{7/12} h_i^{5/3}
=  C \sum_{i=0}^{N-1} l_2 (\sigma l_1)^{1/3} \gamma^{7/12} h_i^{2/3} \\
\label{eq:UBBR}
& = C \sum_{i=0}^{N-1} l_2 (\sigma l_1)^{1/3} \gamma^{7/12} (2^{-i}h_0)^{2/3}
\le C l_2 (\sigma l_1)^{1/3} \gamma^{7/12} h_0^{2/3}.
\end{align}
Equating the energy bounds for $\Omega_{\textrm{B}}$ and $\Omega_{\textrm{BR}}$ gives
\begin{equation*}
l_1 l_2 (\sigma l_1)^{2/3} \gamma^{2/3} h_0^{-2/3} = l_2 (\sigma l_1)^{1/3} \gamma^{7/12} h_0^{2/3}
\quad \Longleftrightarrow \quad h_0 = l_1 \sigma^{1/4} \gamma^{1/16}.
\end{equation*}
(Here we have assumed that the equation $l_1 \sigma \gamma = h_N = 2^{-N} h_0 = 2^{-N} l_1 \sigma^{1/4} \gamma^{1/16}$ has an integer solution $N$.
In general this will not be the case. We obtain the same energy bound, however, by defining $h_0 = l_1 \sigma^{1/4} \gamma^{1/16}$,
$\tilde{N} = \lfloor |\log_2 (l_1 \sigma \gamma/h_0) |\rfloor$, $h_i = 2^{-i} h_0$, $i \in \{1,\ldots,\tilde{N} \}$.)
We remark that $h_0\ge h_N$ because $\gamma\le \sigma^{-4/5}$.
Substituting $h_0 = l_1 \sigma^{1/4} \gamma^{1/16}$ into \eqref{eq:UBB} and \eqref{eq:UBBR} yields
\begin{equation*}
I^{(\sigma,\gamma)}[\bu,w,\Omega_{\textrm{B}}\cup \Omega_{\textrm{BR}}]\le  C l_1 l_2 \sigma^{1/2} \gamma^{5/8}.
\end{equation*}
By combing this and \eqref{eq:UBBL} we find that
\begin{equation*}
I^{(\sigma,\gamma)}[\bu,w,\Omega]\le  C l_1 l_2 (\sigma^{1/2} \gamma^{5/8} + \sigma \gamma).
\end{equation*}
But since $\sigma \gamma \le 1$ and $\gamma \ge 1$
\begin{equation*}
\sigma \gamma = \sigma^{1/2} \gamma^{5/8} ( \sigma \gamma )^{1/2} \gamma^{-1/8} \le \sigma^{1/2} \gamma^{5/8}
\end{equation*}
and we obtain the desired result.

\emph{Case $\gamma < 1$:} In this case we choose $\varepsilon$, $h_N$ and $\delta_N$ as for the case $\gamma = 0$ (see \cite{BenBelgacem1}):
\begin{equation}
\label{hN:g<1}
\varepsilon = h_N = \delta_N = l_1 \sigma.
\end{equation}
(Compare \eqref{hN:g<1} to \eqref{hN:g>1}.) Substituting these into equation  \eqref{B.16}  gives
\begin{equation}
\label{eq:UBBL2}
I^{(\sigma,\gamma)}[\bu,w,\Omega_{\textrm{BL}}] \le C l_1 l_2 \sigma (1 + \gamma) \le  C l_1 l_2 \sigma
\end{equation}
since $\gamma < 1$.
Now we determine $\delta_i$ for $i \in \{ 0 , \ldots , N-1 \}$. In this case
\begin{equation*}
h_N = \sigma l_1  < \sigma l_1 \gamma^{-1/2}.
\end{equation*}
Let $I$ be the largest value of $i<N$ such that $h_i \ge \sigma l_1 \gamma^{-1/2}$
(if there is none, $I=-1$). We have
\begin{equation}
\label{eq:di2}
\delta_i = (\sigma l_1)^{2/3} \gamma^{-1/3} h_i^{1/3} \; \textrm{ for } i \in \{0, \ldots , I \}, \qquad \delta_i = h_i
\; \textrm{ for } i \in \{I+1, \ldots , N-1 \}.
\end{equation}
Therefore
\begin{equation*}
L_i =
\left\{
\begin{array}{ll}
  \max \left\{ h_i , (\sigma l_1)^{-1/3} \gamma^{-1/12} h_i^{4/3} \right\} & \textrm{if }  i \in \{0, \ldots , I \},
 \\
 \max \left\{ h_i , (\sigma l_1)^{-1/2} h_i^{3/2}  \right\} & \textrm{if }  i \in \{I+1, \ldots , N-1 \}.
\end{array}
\right.
\end{equation*}
Since $\gamma < 1$ and $h_i \ge h_N = \sigma l_1 $, it is easy to check that
\begin{equation}
\label{eq:Li2}
L_i =
\left\{
\begin{array}{ll}
  (\sigma l_1)^{-1/3} \gamma^{-1/12} h_i^{4/3}  & \textrm{if }  i \in \{0, \ldots , I \},
 \\
  (\sigma l_1)^{-1/2} h_i^{3/2} & \textrm{if }  i \in \{I+1, \ldots , N-1 \}.
\end{array}
\right.
\end{equation}
It remains to choose $h_0$. We do this as for the case $\gamma \ge 1$ by balancing the energy in the
bulk region $\Omega_{\textrm{B}}$ with the energy in the
branching region $\Omega_{\textrm{BR}}$.
From \eqref{eq:UBB} 
we see that
the energy in the bulk region satisfies
\begin{equation}\label{eq:UBBnn}
I^{(\sigma,\gamma)}[\bu,w,\Omega_{\textrm{B}}] \le
\begin{cases}
C l_1 l_2 (\sigma l_1)^{2/3} \gamma^{2/3} h_0^{-2/3}& \text{ if } h_0\ge \sigma l_1\gamma^{-1/2},\\
 C l_1 l_2 (\sigma l_1)^{2}h_0^{-2}& \text{ if } h_0< \sigma l_1\gamma^{-1/2}.
\end{cases}
\end{equation}
Inserting equations \eqref{eq:di2}, \eqref{eq:Li2} in
 (\ref{eq:ebound}),
\begin{equation*}
I^{(\sigma,\gamma)}[\bu,w,U_i^j\cup V_i^j] \le C
\left\{
\begin{array}{ll}
 (\sigma l_1)^{1/3} \gamma^{7/12} h_i^{5/3} & \textrm{if }  i \in \{0, \ldots , I \},
 \\
 (\sigma l_1)^{-1/2} \gamma h_i^{5/2} + (\sigma l_1)^{3/2} h_i^{1/2} & \textrm{if }  i \in \{I+1, \ldots , N-1 \}.
\end{array}
\right.
\end{equation*}
Therefore the energy in $\Omega_{\textrm{BR}}$ satisfies
\begin{align}
\nonumber
I^{(\sigma,\gamma)}[\bu,w,\Omega_{\textrm{BR}}] & \le C \sum_{i=0}^{I} \frac{l_2}{2h_i} (\sigma l_1)^{1/3} \gamma^{7/12} h_i^{5/3}
+ C \sum_{i=I+1}^{N-1} \frac{l_2}{2h_i} \left(  (\sigma l_1)^{-1/2} \gamma h_i^{5/2} + (\sigma l_1)^{3/2} h_i^{1/2} \right)
\\
\nonumber
& \le C l_2 \left(
(\sigma l_1)^{1/3} \gamma^{7/12} h_0^{2/3} + (\sigma l_1)^{-1/2} \gamma h_{I+1}^{3/2} + (\sigma l_1)^{3/2} h_{N-1}^{-1/2}
\right)
\\
\nonumber
& \le C l_2
\left(
(\sigma l_1)^{1/3} \gamma^{7/12} h_0^{2/3} +\sigma  l_1 \gamma^{1/4} + \sigma l_1 \right)
\end{align}
since $h_{I-1} < \sigma l_1  \gamma^{-1/2}$ and $h_{N-1} =2 h_N = 2\sigma l_1 $. Since $\gamma < 1$,
\begin{equation}
\label{almostthere}
I^{(\sigma,\gamma)}[\bu,w,\Omega_{\textrm{BR}}] \le C  l_2 \left( (\sigma l_1)^{1/3} \gamma^{7/12} h_0^{2/3} + \sigma l_1 \right ).
\end{equation}
It remains to choose $h_0$. If $\gamma< \sigma$, we expect the second term on the right-hand side of \eqref{almostthere} will dominate the energy, therefore
it suffices to choose $h_0$ so that the others are not bigger. Focusing on the second row of  \eqref{eq:UBBnn} one easily sees that
$h_0=\sigma^{1/2}l_1$ is the appropriate choice (this is the same choice used in \cite{BenBelgacem1,BenBelgacem2}). In particular, since $\gamma<\sigma$ this obeys
 $h_0< \sigma l_1\gamma^{-1/2}$, hence is consistent.
Adding together \eqref{eq:UBBnn}, \eqref{eq:UBBL2} and \eqref{almostthere} yields in this case
\begin{equation*}
I^{(\sigma,\gamma)}[\bu,w,\Omega] \le C l_1 l_2 (\sigma  + \sigma^{2/3} \gamma^{7/12})\le  2C l_1 l_2 \sigma.
\end{equation*}
 If instead $\gamma\ge\sigma$, we  choose $h_0$ by equating the $h_0$ term on the right-hand side of \eqref{almostthere} with the first row in the right-hand side of \eqref{eq:UBBnn}. This gives
\begin{equation*}
h_0 = l_1 \sigma^{1/4} \gamma^{1/16}
\end{equation*}
as in the case $\gamma\ge1$. This is consistent since $h_0\ge \sigma l_1\gamma^{-1/2}$ is the same as $\gamma\ge \sigma^{4/3}$, which is true in this case.
Then adding together \eqref{eq:UBBnn}, \eqref{eq:UBBL2} and \eqref{almostthere} yields
\begin{equation*}
I^{(\sigma,\gamma)}[\bu,w,\Omega] \le C l_1 l_2 (\sigma  + \sigma^{1/2} \gamma^{5/8})
\end{equation*}
as required.
\end{proof}
\appendix
\section{Derivation of the Model and Rescaling}
\label{M}
In this appendix we derive the energy \eqref{I.1}.
We model the two-layer material (an elastic film on a substrate) described in the introduction with an energy consisting of two parts, an elastic energy for the thin film and a bonding energy for the interaction between the film and the substrate. We take the elastic energy to be the von K\'arm\'an energy, which penalises extension (stretching and compression) and bending:
\beqn{M.1}
\begin{aligned}
\hat I_{\mathrm{vK}} & := \frac{Eh_\mathrm{f}}{2(1-\nu^2)} \int_{\Omega} \left\{ (1-\nu) \left| \frac{D \bu + (D \bu)^{\textrm{T}}}{2} + \frac{D w \otimes D w}{2} - \varepsilon_* \Id \right|^2
\right.
\\ & \left. \phantom{\frac{Eh_\mathrm{f}}{2(1-\nu^2)} \int_{\Omega}} \quad + \nu \left| \mathrm{tr} \left(
\frac{D \bu + (D \bu)^{\textrm{T}}}{2} + \frac{D w \otimes D w}{2} - \varepsilon_* \Id \right) \right|^2 \right\} \, d\bx
\\
& \phantom{==} + \frac{E h_\mathrm{f}^3}{24(1-\nu^2)} \int_{\Omega} \left\{ (1-\nu) | D^2 w |^2 + \nu (\Delta w)^2 \right\} \, d\bx
\end{aligned}
\eeq
where $\Omega = (0,l_1) \times (0,l_2)$ is the set of material points $(x,y)$ of the film, $h_\textrm{f}$ is the thickness of the film, $E$ is the Young's Modulus, and $\nu$ is the Poisson ratio. The functions $\bu(x,y) = (u(x,y),v(x,y))$ and  $w(x,y)$ are the in-plane and transversal (vertical) displacements of the film from the isotropically compressed state
$((1-\varepsilon_*)x,(1-\varepsilon_*)y,0)$, where $0<\varepsilon_*<1$ is the compression ratio. Note that $\e_*=0$ corresponds to no compression and $\e_*=1$ corresponds to total compression. Thus $(u,v,w)=(0,0,0)$ corresponds to the isotropically compressed state of the film and
$(u,v,w)=(\varepsilon_* x,\varepsilon_* y,0)$ corresponds to the relaxed, natural state. The substrate is taken to be at height $z=0$. Therefore the transversal displacement $w$ must satisfy the constraint $w \ge 0$ (since the film cannot go below the substrate). If $w(x,y)=0$ then the material point $(x,y)$ of the film is bonded down to the substrate.

For the energy scaling analysis that we perform we can set the Poisson ratio $\nu$ equal to zero without loss of generality since the terms of \eqref{M.1} with factor $\nu$ can be bound from above and below by those without, as shown in \cite[Appendix B]{BenBelgacem1}.

The interfacial force between the thin film and the substrate is an attractive van der Waals force.
We model it in a simple way with a bonding energy $\hat I_{\textrm{Bo}}$ that penalises debonding from the substrate:
\beqn{M.2}
\hat I_{\textrm{Bo}} := \gamma^* | \{ (x,y)\in \Omega \, : \, w(x,y)>0 \} |  
\eeq
where 
$\gamma^*$ is a constant depending on the material properties of the film and the substrate. In more sophisticated models $I_{\textrm{Bo}}$ would also depend
on the size of $w$, i.e., how far the film is from the substrate.

Thus the total energy we assign to the system is
\beqn{M.3}
\hat I := \hat I_{\textrm{vK}} + \hat I_{\textrm{Bo}}.
\eeq
On the boundary $\{x=0\}$ we assume that the film is fixed to the substrate and satisfies the clamped boundary conditions
\beqn{M.4}
u(0,y)=0, \quad v(0,y)=0, \quad w(0,y) = 0, \quad D w (0,y) = 0.
\eeq
(Note that in the experiments of \cite{MeiThurmer2007} the film is actually fixed to a buffer layer and satisfies slightly different boundary conditions. This is discussed in Appendix \ref{H}.) The film is free on the rest of its boundary, $\{x=l_1\} \cup \{y=0\} \cup \{y=l_2\}$.
\paragraph{Rescaling.} In order to reduce the number of parameters we define rescaled variables, denoted with a superscript $*$, by
\begin{gather}
\label{M.5}
\bu =: 2 \varepsilon_*  \bu^*, \qquad w =: (2 \varepsilon_*)^{1/2} \, w^*.
\end{gather}
By substituting from equation \eqref{M.5} into equation \eqref{M.3}, multiplying by
$2 (1-\nu^2) / (E h_\mathrm{f} \varepsilon_*^2)$, setting $\nu=0$, and dropping all the superscripts * for convenience, we obtain a new energy $I^{(\sigma,\g)}$:
\beqn{M.6}
\begin{aligned}
I^{(\sigma,\g)}[\bu,w,l_1,l_2]
& = \int_{x=0}^{l_1} \! \int_{y=0}^{l_2} |D \bu + (D \bu)^\mathrm{T} + D w \otimes D w - \Id |^2 \, dx \, dy \\
& \phantom{=} + (\sigma l_1)^2 \int_{x=0}^{l_1} \! \int_{y=0}^{l_2} | D^2 w|^2 \, dx \, dy  + \g | \{ (x,y)\in \Omega \, : \, w(x,y)>0 \} | \\
\end{aligned}
\eeq
where $\sigma$ is the rescaled film thickness and $\g$ is the rescaled bonding energy per unit area:
\beqn{M.7}
\sigma := \frac{h_\mathrm{f}}{l_1 (6 \varepsilon_*)^{1/2}}, \qquad \g :=\frac{2 (1-\nu^2) \g^*}{E h_\mathrm{f} \varepsilon_*^2} = \frac{2 \g^*}{E h_\mathrm{f} \varepsilon_*^2} \quad (\textrm{since } \nu=0).
\eeq
Throughout this paper we refer to $\g$ as the bonding strength, although note that it also depends on the unscaled film thickness $h_{\textrm{f}}$. Equation \eqref{M.6} is exactly equation \eqref{I.1} given in the Introduction.

\section{Upper Bounds for Buffer Layers with Nonzero Thickness}
\label{H}
In the experiments of \cite{MeiThurmer2007} the film is not clamped to the substrate, but rather to a thin buffer layer, and satisfies clamped boundary conditions of the form
\begin{equation}
u(0,y)=0, \quad v(0,y)=0, \quad w(0,y) = h_\mathrm{b}, \quad D w (0,y) = \mathbf{0}
\end{equation}
where $h_\mathrm{b}>0$ is the thickness of the buffer layer.
In this paper we considered the unphysical case $h_\mathrm{b}=0$ in order to simplify the analysis.
In this appendix we show that the upper bounds of Theorem \ref{UB1} also hold when $h_\mathrm{b}$ is sufficiently small.
Recall that $\Omega = (0,l_1)\times(0,l_2)$.
Define
\begin{equation}
\nonumber
V_{h_{\textrm{b}}} = \left\{
(\bu,w) \in H^1(\Omega;\R^2) \times H^2(\Omega) : w= h_{\textrm{b}}, \, \bu = D w = \mathbf{0}, \,  \textrm{ on } \{ 0 \} \times (0,l_2), \; w \ge 0
\right\}.
\end{equation}
\begin{Theorem}[Upper Bounds for $h_{\textrm{b}} >0$] \label{UB4}
Let $0<l_1\le l_2$, $\sigma\in(0,1)$, $\g\ge 0$.
Assume that $h_{\textrm{b}}$ satisfies
\beq
\nonumber
0 \le \frac{h_{\textrm{b}}}{l_1} \le \min \{ \sigma,\sigma^{-1} \g^{-3/2} \}.
\eeq
Then there exists admissible displacement fields $(\bu,w) \in V_{h_{\textrm{b}}}$ satisfying the same upper bounds as in Theorem \ref{UB1}.
\end{Theorem}
\begin{proof}

We construct displacement fields satisfying the required upper bounds by simply appending  to the constructions used to prove Theorem \ref{UB1} a boundary layer in which the film is bent upwards from height $0$ to height $h_{\textrm{b}}$. Let $\eta>0$ be the boundary layer thickness. For $x \in [0,\eta]$, $y \in [0,l_2]$ define
\beqn{H.1}
u = \frac12 x, \quad v = 0, \quad w = h_{\textrm{b}} \left[ 1 - \psi \left( \frac{x}{\eta} \right) \right]
\eeq
where $\psi\in C^2(\R)$ is a function with $\psi(t)=0$ for $t\le 0$, $\psi(t)=1$ for $t\ge 1$. For $x \ge \eta$, define $u$, $v$, $w$ as in the proof of Theorem \ref{UB1} (except that $u$ is replaced with $u + \eta/2$). The energy $E_{\textrm{BL}}^\eta$ of $(u,v,w)$ in the boundary layer $(0,\eta)\times(0,l_2)$ satisfies
\beqn{H.2}
E_{\textrm{BL}}^\eta  \le l_2 \eta \left( \frac{h_{\textrm{b}}^4}{\eta^4} + (\sigma l_1)^2 \frac{h_{\textrm{b}}^2}{\eta^4} + 1 + \gamma \right)
 \le l_2 \eta \left(2(\sigma l_1)^2 \frac{h_{\textrm{b}}^2}{\eta^4} + 1 + \gamma \right)
\eeq
since we assumed that $h_{\textrm{b}} \le \sigma l_1$. We choose
\beqn{H.3}
\eta = (\sigma l_1 h_{\textrm{b}})^{1/2}\max\{1,\gamma\}^{-1/4}
\eeq
so that
\beqn{H.4}
E_{\textrm{BL}}^\eta \le l_2 (\sigma l_1 h_{\textrm{b}})^{1/2}\max\{1,\gamma\}^{3/4}.
\eeq
For each regime it is easy to check that this boundary layer energy is no greater than the energy of the constructions given in Theorem \ref{UB1}.
\end{proof}
This is the simplest possible modification of the proof of Theorem \ref{UB1} and we do not claim that the bound
$h_{\textrm{b}}/l_1 < \min \{\sigma,\sigma^{-1} \g^{-3/2}\}$ is sharp.


\section{Poincar\'e Constant}
\label{App:P}
In this section we prove the Poincar\'e inequality \eqref{eq:Poin}. It is sufficient to prove the following:
\begin{Lemma}
Let $Q=(0,1) \times (0,1)$ be the unit square and $\bof \in C^2(Q;\mathbb{R}^2)$. Let $\bof$ be zero on at least half of the square:
\begin{equation*}
| \{ \bx \in Q : \bof(\bx) = \mathbf{0} \}| \ge \frac12.
\end{equation*}
Then
\begin{equation*}
\int_Q |\bof|^2 \, d \bx \le \int_Q |D \bof|^2 \, d \bx.
\end{equation*}
\end{Lemma}
\begin{proof}
First we recall the Poincar\'e inequality in one dimension. Let $\bg : [0,1] \to \mathbb{R}^2$. Write $\bg$ as the Fourier cosine series
(the Fourier series of the even extension of $g$ to $[-1,1]$)
\begin{equation*}
\bg(x) = \frac{\hat{\bg}_0}{2} + \sum_{k=1}^\infty \hat{\bg}_k \cos(\pi k x), \qquad \hat{\bg}_k = 2 \int_0^1 \bg(x) \cos(\pi k x) \, dx.
\end{equation*}
Let $\overline{\bg} = \hat{\bg}_0/2 = \int_0^1 \bg(x) \, dx$.
By Parseval's Theorem
\begin{equation}
\label{P1D}
\int_0^1 |\bg(x) - \overline{\bg} |^2 \, dx = \frac12 \sum_{k=1}^\infty |\hat{\bg}_k|^2 = \frac{1}{2 \pi^2} \sum_{k=1}^\infty \frac{1}{k^2} |\pi k \hat{\bg}_k|^2
\le \frac{1}{2 \pi^2} \sum_{k=1}^\infty |\pi k \hat{\bg}_k|^2 = \frac{1}{\pi^2} \int_0^1 |\bg'(x)|^2 \, dx,
\end{equation}
which is the Poincar\'e inequality in one dimension. Let
\begin{equation*}
\boa(x) = \int_0^1 \bof(x,y) \, dy, \qquad \overline{\boa} = \int_0^1 \boa(x) \, dx = \int_Q \bof(x,y) \, dxdy.
\end{equation*}
By equation  \eqref{P1D}
\begin{gather*}
\int_0^1 |\boa(x) - \overline{\boa}|^2 \, dx \le \frac{1}{\pi^2} \int_0^1 |\boa'(x)|^2 \, dx \le \frac{1}{\pi^2} \int_0^1 \int_0^1 |\bof_x |^2 \, dxdy, \\
\int_0^1 |\bof(x,y) - \boa(x)|^2 \, dy \le \frac{1}{\pi^2} \int_0^1 |\bof_y|^2 \, dy.
\end{gather*}
Therefore
\begin{equation}
\label{est1}
\int_Q | \bof - \overline{\boa} |^2 \, dxdy \le 2
\left( \int_Q | \bof - \boa |^2 \, dxdy + \int_Q | \boa - \overline{\boa} |^2 \, dxdy \right)
\le \frac{2}{\pi^2} \int_Q |D \bof|^2 \, dxdy.
\end{equation}
Let $U = \{ \bx \in Q : \bof(\bx) = \mathbf{0} \}$. We have
\begin{equation}
\label{est2}
\frac12 |\overline{\boa}|^2 \le |U| |\overline{\boa}|^2 = \int_U |\bof - \overline{\boa}|^2 \, dx \le \frac{2}{\pi^2} \int_Q |D \bof|^2 \, dxdy.
\end{equation}
Observe that
\begin{equation*}
\int_Q |\bof - \overline{\boa}|^2 \, d \bx = \int_Q |\bof|^2 \, d\bx - |\overline{\boa}|^2.
\end{equation*}
By combining this with estimates \eqref{est1} and \eqref{est2} we complete the proof:
\begin{equation*}
\int_Q |\bof|^2 \, d \bx = \int_Q |\bof - \overline{\boa}|^2 \, d \bx + |\overline{\boa}|^2
\le \frac{6}{\pi^2} \int_Q |D \bof|^2 \, d \bx.
\end{equation*}
\end{proof}

%
%
\paragraph{Acknowledgements.}
The authors would like to thank O.~G.~Schmidt and other members of the Institute for Integrative Nanosciences, IFW Dresden (including P.~Cendula, S.~Kiravittaya and Y.~F.~Mei) for interesting discussions about the experiments that were the motivation for this paper.
This work was partially supported by the Deutsche Forschungsgemeinschaft
through the  Sonderforschungsbereich 1060 {\em ``The mathematics of emergent effects''}.

%
%
\bibliographystyle{amsplain}
\bibliography{delamin}

\end{document}